\documentclass[runningheads, a4paper]{llncs}

\usepackage[utf8]{inputenc}
\usepackage[T1]{fontenc}
\usepackage[british]{babel}

\usepackage{textcomp}
\usepackage{scrhack}
\usepackage{xspace}
\usepackage{mparhack}
\usepackage{fixltx2e}

\usepackage{amsmath}
\usepackage[warning, all]{onlyamsmath}\AtBeginDocument{\catcode`\$=3}
\usepackage{amssymb}
\usepackage{stmaryrd}
\usepackage{dsfont}
\usepackage{mathtools}
\usepackage{mathdots}
\usepackage{gensymb}
\allowdisplaybreaks

\usepackage{graphicx}
\usepackage{subfig}
\usepackage{scalerel}
\usepackage{csquotes}
\usepackage{url}
\usepackage{adjustbox}
\usepackage{tabularx}

\usepackage{caption}
\usepackage{tikz}
\usetikzlibrary{positioning, arrows, patterns} 

\spnewtheorem{main-theorem}[theorem]{Main Theorem}{\bfseries}{\itshape}
\spnewtheorem{intuition}{Intuition}{\bfseries}{\itshape}

\urldef{\mails}\path|simon.wacker@kit.edu|
\newcommand{\keywords}[1]{\par\addvspace\baselineskip\noindent\keywordname\enspace\ignorespaces#1}

\newcommand*{\define}[1]{\emph{#1}}

\providecommand\iff{\DOTSB\;\Longleftrightarrow\;}
\providecommand\implies{\DOTSB\;\Longrightarrow\;}

\DeclarePairedDelimiter\parens{\lparen}{\rparen}

\DeclarePairedDelimiter\abs{\lvert}{\rvert}
\DeclarePairedDelimiter\norm{\lVert}{\rVert}
\DeclarePairedDelimiterX\inner[2]{\langle}{\rangle}{#1,#2} 
\DeclarePairedDelimiter\set{\{}{\}} 
\DeclarePairedDelimiter\net{\{}{\}}
\DeclarePairedDelimiter\family{\{}{\}}

\mathchardef\breakingcomma\mathcode`\,
{\catcode`,=\active
  \gdef,{\breakingcomma\discretionary{}{}{}}
}
\newcommand*\ntuple[1]{\lparen\mathcode`\,=\string"8000 #1\rparen}

\newcommand*{\N}{\mathbb{N}}

\newcommand*{\R}{\mathbb{R}}

\newcommand*{\K}{\mathbb{K}}

\DeclareMathOperator{\boundeds}{\ell^\infty} 
\DeclareMathOperator{\simples}{\mathcal{E}} 
\DeclareMathOperator{\pmeasures}{\mathcal{PM}}
\DeclareMathOperator{\means}{\mathcal{M}}

\DeclareMathOperator{\automorphisms}{Aut}

\newcommand*{\unityfnc}{\mathds{1}} 
\newcommand*{\nullityfnc}{0} 

\DeclareMathOperator{\identity}{id}

\DeclareMathOperator{\Exists}{\exists}
\DeclareMathOperator{\ForEach}{\forall}
\newcommand*\Holds{:}
\newcommand*\SuchThat{:}

\newcommand*\leftaction{\triangleright}

\newcommand*\rightsemiaction{\mathbin{\protect\scalerel*{\trianglelefteqslant}{\rhd}}} 

\makeatletter
\providecommand*{\Dashv}{%
  \mathrel{%
    \mathpalette\@Dashv\vDash
  }%
}
\newcommand*{\@Dashv}[2]{%
  \reflectbox{$\m@th#1#2$}%
}
\makeatother
\newcommand*\measureamact{\mathbin{\protect\scalerel*{\Dashv}{\rhd}}}
\newcommand*\measureamactleft{\mathbin{\protect\scalerel*{\vDash}{\rhd}}}

\newcommand{\VDash}{%
  \mathrel{
    \text{\clipbox{0pt 0pt {.8\width} 0pt}{$\Vdash$}}
    \mkern.9mu
    \text{\adjustbox{width=.87\width,height=\height}{$\vDash$}}
  }
}
\makeatletter
\providecommand*{\DashV}{%
  \mathrel{%
    \mathpalette\@DashV\VDash
  }%
}
\newcommand*{\@DashV}[2]{%
  \reflectbox{$\m@th#1#2$}%
}
\makeatother
\makeatletter
\providecommand*{\dashV}{%
  \mathrel{%
    \mathpalette\@dashV\Vdash
  }%
}
\newcommand*{\@dashV}[2]{%
  \reflectbox{$\m@th#1#2$}%
}
\makeatother
\newcommand*\funcamact{\mathbin{\protect\scalerel*{\dashV}{\rhd}}}
\newcommand*\meanamact{\mathbin{\protect\scalerel*{\DashV}{\rhd}}}
\newcommand*\funcamactleft{\mathbin{\protect\scalerel*{\Vdash}{\rhd}}}
\newcommand*\meanamactleft{\mathbin{\protect\scalerel*{\VDash}{\rhd}}}

\DeclareMathOperator{\Sym}{Sym}

\DeclareMathOperator{\vecspan}{span}

\DeclareMathOperator{\ev}{ev}

\makeatletter
\def\moverlay{\mathpalette\mov@rlay}
\def\mov@rlay#1#2{\leavevmode\vtop{%
   \baselineskip\z@skip \lineskiplimit-\maxdimen
   \ialign{\hfil$\m@th#1##$\hfil\cr#2\crcr}}}
\newcommand{\charfusion}[3][\mathord]{
    #1{\ifx#1\mathop\vphantom{#2}\fi
        \mathpalette\mov@rlay{#2\cr#3}
      }
    \ifx#1\mathop\expandafter\displaylimits\fi}
\makeatother
\newcommand{\cupdot}{\charfusion[\mathbin]{\cup}{\cdot}}
\newcommand{\bigcupdot}{\charfusion[\mathop]{\bigcup}{\cdot}}

\DeclareMathOperator{\powerset}{\mathcal{P}}

\newcommand*\suchthat{\mid} 
\newcommand*\from{\colon} 
\newcommand*\after{\circ}
\newcommand*\isomorphic{\cong}
\newcommand*\quotient{\slash}
\renewcommand*{\restriction}{\mathord{\upharpoonright}}

\newcommand*{\blank}{\mathord{\_}} 
\newcommand{\closure}[1]{\mkern 1.5mu\overline{\mkern-1.5mu#1\mkern-1.5mu}\mkern 1.5mu} 

\newcommand*{\graffito}[1]{}
\newcommand*{\mathnote}[1]{}
\renewcommand*{\index}[1]{}

\begin{document}

  \mainmatter

  \title{Right Amenable Left Group Sets and the Tarski-Følner Theorem}
  \authorrunning{Right Amenable Left Group Sets and the Tarski-Følner Theorem}
  \author{Simon Wacker}

  \institute{%
    Karlsruhe Institute of Technology\\
    \mails\\
    \url{http://www.kit.edu}%
  }

  \maketitle

  \begin{abstract}
    We introduce right amenability, right Følner nets, and right paradoxical decompositions for left homogeneous spaces and prove the Tarski-Følner theorem for left homogeneous spaces with finite stabilisers. It states that right amenability, the existence of right Følner nets, and the non-existence of right paradoxical decompositions are equivalent.
  \end{abstract}

  The notion of amenability for groups was introduced by John von Neumann in 1929, see the paper \enquote{Zur allgemeinen Theorie des Maßes}\cite{von-neumann:1929}. It generalises the notion of finiteness. A group $G$ is \emph{left} or \emph{right amenable} if there is a finitely additive probability measure on $\powerset(G)$ that is invariant under left and right multiplication respectively. Groups are left amenable if and only if they are right amenable. A group is \emph{amenable} if it is left or right amenable.

  The definitions of left and right amenability generalise to left and right group sets respectively. A left group set $\ntuple{M, G, \leftaction}$ is \emph{left amenable} if there is a finitely additive probability measure on $\powerset(M)$ that is invariant under $\leftaction$. There is in general no natural action on the right that is to a left group action what right multiplication is to left group multiplication. Therefore, for a left group set there is no natural notion of right amenability.

  A transitive left group action $\leftaction$ of $G$ on $M$ induces, for each element $m_0 \in M$ and each family $\family{g_{m_0, m}}_{m \in M}$ of elements in $G$ such that, for each point $m \in M$, we have $g_{m_0, m} \leftaction m_0 = m$, a right quotient set semi-action $\rightsemiaction$ of $G \quotient G_0$ on $M$ with defect $G_0$ given by $m \rightsemiaction g G_0 = g_{m_0, m} g g_{m_0, m}^{-1} \leftaction m$, where $G_0$ is the stabiliser of $m_0$ under $\leftaction$. Each of these right semi-actions is to the left group action what right multiplication is to left group multiplication. They occur in the definition of global transition functions of cellular automata over left homogeneous spaces as defined in \cite{wacker:automata:2016}. A \emph{coordinate system} is a choice of $m_0$ and $\family{g_{m_0, m}}_{m \in M}$. 

  A left homogeneous space is \emph{right amenable} if there is a coordinate system such that there is a finitely additive probability measure on $\powerset(M)$ that is semi-invariant under $\rightsemiaction$. For example finite left homogeneous spaces, abelian groups, and finitely right generated left homogeneous spaces of sub-exponential growth are right amenable, in particular, quotients of finitely generated groups of sub-exponential growth by finite subgroups acted on by left multiplication.

  A net of non-empty and finite subsets of $M$ is a \emph{right Følner net} if, broadly speaking, these subsets are asymptotically invariant under $\rightsemiaction$. A finite subset $E$ of $G \quotient G_0$ and two partitions $\family{A_e}_{e \in E}$ and $\family{B_e}_{e \in E}$ of $M$ constitute a \emph{right paradoxical decomposition} if the map $\blank \rightsemiaction e$ is injective on $A_e$ and $B_e$, and the family $\family{(A_e \rightsemiaction e) \cupdot (B_e \rightsemiaction e)}_{e \in E}$ is a partition of $M$. The Tarski-Følner theorem states that right amenability, the existence of right Følner nets, and the non-existence of right paradoxical decompositions are equivalent.

  The Tarski alternative theorem and the theorem of Følner, which constitute the Tarski-Følner theorem, are famous theorems by Alfred Tarski and Erling Følner from 1938 and 1955, see the papers \enquote{Algebraische Fassung des Maßproblems}\cite{tarski:1938} and \enquote{On groups with full Banach mean value}\cite{folner:1955}. This paper is greatly inspired by the monograph \enquote{Cellular Automata and Groups}\cite{ceccherini-silberstein:coornaert:2010} by Tullio Ceccherini-Silberstein and Michel Coornaert. 

  For a right amenable left homogeneous space with finite stabilisers we may choose a right Følner net. Using this net we show in \cite{wacker:garden:2016} that the Garden of Eden theorem holds for such spaces. It states that a cellular automaton with finite set of states and finite neighbourhood over such a space is surjective if and only if it is pre-injective.

  In Sect.~\ref{sec:measures-and-means} we introduce finitely additive probability measures and means, and kind of right semi-actions on them. In Sect.~\ref{sec:right-amenability} we introduce right amenability. In Sect.~\ref{sec:right-folner-nets} we introduce right Følner nets. In Sect.~\ref{sec:right-paradoxical-decomposition} we introduce right paradoxical decompositions. In Sect.~\ref{sec:tarski-folner-theorem} we prove the Tarski alternative theorem and the theorem of Følner. And in Sect.~\ref{sec:left-vs-right} we show under which assumptions left implies right amenability and give two examples of right amenable left homogeneous spaces.

  \subsubsection{Preliminary Notions.} A \define{left group set} is a triple $\ntuple{M, G, \leftaction}$, where $M$ is a set, $G$ is a group, and $\leftaction$ is a map from $G \times M$ to $M$, called \define{left group action of $G$ on $M$}, such that $G \to \Sym(M)$, $g \mapsto [g \leftaction \blank]$, is a group homomorphism. The action $\leftaction$ is \define{transitive} if $M$ is non-empty and for each $m \in M$ the map $\blank \leftaction m$ is surjective; and \define{free} if for each $m \in M$ the map $\blank \leftaction m$ is injective. For each $m \in M$, the set $G \leftaction m$ is the \define{orbit of $m$}, the set $G_m = (\blank \leftaction m)^{-1}(m)$ is the \define{stabiliser of $m$}, and, for each $m' \in M$, the set $G_{m, m'} = (\blank \leftaction m)^{-1}(m')$ is the \define{transporter of $m$ to $m'$}.

  A \define{left homogeneous space} is a left group set $\mathcal{M} = \ntuple{M, G, \leftaction}$ such that $\leftaction$ is transitive. A \define{coordinate system for $\mathcal{M}$} is a tuple $\mathcal{K} = \ntuple{m_0, \family{g_{m_0, m}}_{m \in M}}$, where $m_0 \in M$ and, for each $m \in M$, we have $g_{m_0, m} \leftaction m_0 = m$. The stabiliser $G_{m_0}$ is denoted by $G_0$. The tuple $\mathcal{R} = \ntuple{\mathcal{M}, \mathcal{K}}$ is a \define{cell space}. The set $\set{g G_0 \suchthat g \in G}$ of left cosets of $G_0$ in $G$ is denoted by $G \quotient G_0$. The map $\rightsemiaction \from M \times G \quotient G_0 \to M$, $(m, g G_0) \mapsto g_{m_0, m} g \leftaction m_0$ is a \define{right semi-action of $G \quotient G_0$ on $M$ with defect $G_0$}, which means that
  \begin{gather*}
    \ForEach m \in M \Holds m \rightsemiaction G_0 = m,\\
    \ForEach m \in M \ForEach g \in G \Exists g_0 \in G_0 \SuchThat \ForEach \mathfrak{g}' \in G \quotient G_0 \Holds
          m \rightsemiaction g \cdot \mathfrak{g}' = (m \rightsemiaction g G_0) \rightsemiaction g_0 \cdot \mathfrak{g}'.
  \end{gather*}
  It is \define{transitive}, which means that the set $M$ is non-empty and for each $m \in M$ the map $m \rightsemiaction \blank$ is surjective; and \define{free}, which means that for each $m \in M$ the map $m \rightsemiaction \blank$ is injective; and \define{semi-commutes with $\leftaction$}, which means that
  \begin{equation*}
    \ForEach m \in M \ForEach g \in G \Exists g_0 \in G_0 \SuchThat \ForEach \mathfrak{g}' \in G \quotient G_0 \Holds
          (g \leftaction m) \rightsemiaction \mathfrak{g}' = g \leftaction (m \rightsemiaction g_0 \cdot \mathfrak{g}').
  \end{equation*}
  For each $A \subseteq M$, let $\unityfnc_A \from M \to \set{0, 1}$ be the indicator function of $A$.

  \section{Finitely Additive Probability Measures and Means}
  \label{sec:measures-and-means}

  In this section, let $\mathcal{R} = \ntuple{\ntuple{M, G, \leftaction}, \ntuple{m_0, \family{g_{m_0, m}}_{m \in M}}}$ be a cell space.


  \begin{definition} 
    Let $\mu \from \powerset(M) \to [0,1]$ be a map. It is called
    \begin{enumerate}
      \item \define{normalised}\graffito{normalised} if and only if $\mu(M) = 1$;
      \item \define{finitely additive}\graffito{finitely additive} if and only if, 
            \begin{equation*}
              \ForEach A \subseteq M \ForEach B \subseteq M \Holds \parens[\big]{A \cap B = \emptyset \implies \mu(A \cup B) = \mu(A) + \mu(B)};
            \end{equation*}
      \item \define{finitely additive probability measure on $M$}\graffito{finitely additive probability measure $\mu$ on $M$} if and only if it is normalised and finitely additive.
    \end{enumerate}
    The set of all finitely additive probability measures on $M$ is denoted by $\pmeasures(M)$\graffito{$\pmeasures(M)$}.
  \end{definition}

  \begin{definition}
    The group $G$ acts on $[0,1]^{\powerset(M)}$ on the left by
    \begin{align*} 
      \measureamactleft \from G \times [0,1]^{\powerset(M)} &\to     [0,1]^{\powerset(M)}, \mathnote{$\measureamactleft$}\\
                                               (g, \varphi) &\mapsto [A \mapsto \varphi(g^{-1} \leftaction A)],
    \end{align*}
    such that $G \measureamactleft \pmeasures(M) \subseteq \pmeasures(M)$. 
  \end{definition}


  \begin{definition} 
    The quotient set $G \quotient G_0$ kind of semi-acts on $[0,1]^{\powerset(M)}$ on the right by
    \begin{align*}
      \measureamact \from [0,1]^{\powerset(M)} \times G \quotient G_0 &\to     [0,1]^{\powerset(M)}, \mathnote{$\measureamact$}\\
                                              (\varphi, \mathfrak{g}) &\mapsto [A \mapsto \varphi(A \rightsemiaction \mathfrak{g})].
    \end{align*}
  \end{definition}



  \begin{definition}
    Let $\varphi$ be an element of $[0,1]^{\powerset(M)}$. It is called \define{$\measureamact$-semi-invariant}\graffito{$\measureamact$-semi-invariant} if and only if, for each element $\mathfrak{g} \in G \quotient G_0$ and each subset $A$ of $M$ such that the map $\blank \rightsemiaction \mathfrak{g}$ is injective on $A$, we have $(\varphi \measureamact \mathfrak{g})(A) = \varphi(A)$.
  \end{definition}

  \begin{remark} 
  \label{rem:groups:measureamact}
    Let $\mathcal{R}$ be the cell space $\ntuple{\ntuple{G, G, \cdot}, \ntuple{e_G, \family{g}_{g \in G}}}$. Then, $G_0 = \set{e_G}$ and $\rightsemiaction = \cdot$. Hence, $\measureamact \from (\varphi, g) \mapsto [A \mapsto \varphi(A \cdot g)]$. Except for $g$ not being inverted, this is the right group action of $G$ on $\pmeasures(M)$ as defined in \cite[Sect.~4.3, Paragraph~4]{ceccherini-silberstein:coornaert:2010}. Moreover, for each element $g \in G$, the map $\blank \rightsemiaction g$ is injective. Hence, being $\measureamact$-semi-invariant is the same as being right-invariant as defined in \cite[Sect.~4.4, Paragraph~2]{ceccherini-silberstein:coornaert:2010}. 
  \end{remark}

  \begin{definition} 
    The vector space of bounded real-valued functions on $M$ with pointwise addition and scalar multiplication is denoted by $\boundeds(M)$,
    the supremum norm on $\boundeds(M)$ is denoted by $\norm{\blank}_\infty$,
    the topological dual space of $\boundeds(M)$ is denoted by $\boundeds(M)^*$,
    the pointwise partial order on $\boundeds(M)$ is denoted by $\leq$,
    and the constant function $[m \mapsto 0]$ is denoted by $\nullityfnc$. 
  \end{definition}



  \begin{definition} 
  \label{def:mean}
    Let $\nu \from \boundeds(M) \to \R$ be a map. It is called
    \begin{enumerate}
      \item \define{normalised}\graffito{normalised} if and only if $\nu(\unityfnc_M) = 1$;
      \item \define{non-negativity preserving} if and only if
            \begin{equation*}
              \ForEach f \in \boundeds(M) \Holds (f \geq \nullityfnc \implies \nu(f) \geq 0); \mathnote{non-negativity preserving}
            \end{equation*}
      \item \define{mean on $M$}\graffito{mean $\nu$ on $M$} if and only if it is linear, normalised, and non-negativity preserving. 
    \end{enumerate}
    The set of all means on $M$ is denoted by $\means(M)$\graffito{$\means(M)$}. 
  \end{definition}

  %

  \begin{definition} 
    Let $\Psi$ be a map from $\boundeds(M)$ to $\boundeds(M)$. It is called \define{non-negativity preserving}\graffito{non-negativity preserving} if and only if 
    \begin{equation*}
      \ForEach f \in \boundeds(M) \Holds (f \geq \nullityfnc \implies \Psi(f) \geq 0).
    \end{equation*}
  \end{definition}

  \begin{lemma} 
  \label{lem:liberation-preimage}
    Let $G_0$ be finite, let $A$ be a finite subset of $M$, and let $\mathfrak{g}$ be an element of $G \quotient G_0$. Then, $\abs{(\blank \rightsemiaction \mathfrak{g})^{-1}(A)} \leq \abs{G_0} \cdot \abs{A}$.
  \end{lemma}

  \begin{proof}
    Let $a \in A$ such that $(\blank \rightsemiaction \mathfrak{g})^{-1}(a) \neq \emptyset$. There are $m$ and $m' \in M$ such that $G_{m_0,m} = \mathfrak{g}$ and $m' \rightsemiaction \mathfrak{g} = a$. For each $m'' \in M$, we have $m'' \rightsemiaction \mathfrak{g} = g_{m_0, m''} \leftaction m$ and hence 
    \begin{align*} 
      m'' \rightsemiaction \mathfrak{g} = a
      &\iff m'' \rightsemiaction \mathfrak{g} = m' \rightsemiaction \mathfrak{g}\\
      &\iff g_{m_0, m'}^{-1} g_{m_0, m''} \leftaction m = m\\
      &\iff g_{m_0, m'}^{-1} g_{m_0, m''} \in G_m\\
      &\iff g_{m_0, m''} \in g_{m_0, m'} G_m.
    \end{align*}
    Moreover, for each $m''$ and each $m''' \in M$ with $m'' \neq m'''$, we have $g_{m_0, m''} \neq g_{m_0, m'''}$. 
    Thus,
    \begin{align*}
      \abs{(\blank \rightsemiaction \mathfrak{g})^{-1}(a)}
      &=    \abs{\set{m'' \in M \suchthat m'' \rightsemiaction \mathfrak{g} = a}}\\
      &=    \abs{\set{m'' \in M \suchthat g_{m_0, m''} \in g_{m_0, m'} G_m}}\\
      &\leq \abs{g_{m_0, m'} G_m}\\
      &=    \abs{G_m}\\
      &=    \abs{G_0}.
    \end{align*}
    Therefore, because $(\blank \rightsemiaction \mathfrak{g})^{-1}(A) = \bigcup_{a \in A} (\blank \rightsemiaction \mathfrak{g})^{-1}(a)$, we have $\abs{(\blank \rightsemiaction \mathfrak{g})^{-1}(A)} \leq \abs{G_0} \cdot \abs{A}$. \qed
  \end{proof}

  \begin{definition}
    The group $G$ acts on $\boundeds(M)$ on the left by
    \begin{align*}
      \funcamactleft \from G \times \boundeds(M) &\to     \boundeds(M), \mathnote{$\funcamactleft$}\\
                                          (g, f) &\mapsto [m \mapsto f(g^{-1} \leftaction m)].
    \end{align*}
  \end{definition}


  \begin{lemma}
  \label{lem:kind-of-semi-action-of-quotient-on-boundeds}
    Let $G_0$ be finite. The quotient set $G \quotient G_0$ kind of semi-acts on $\boundeds(M)$ on the right by
    \begin{align*}
      \funcamact \from \boundeds(M) \times G \quotient G_0 &\to     \boundeds(M), \mathnote{$\funcamact$}\\
                                         (f, \mathfrak{g}) &\mapsto [m \mapsto \sum_{m' \in (\blank \rightsemiaction \mathfrak{g})^{-1}(m)} f(m')],
    \end{align*}
    such that, for each tuple $(f, \mathfrak{g}) \in \boundeds(M) \times G \quotient G_0$, we have $\norm{f \funcamact \mathfrak{g}}_\infty \leq \abs{G_0} \cdot \norm{f}_\infty$.
  \end{lemma}

  \begin{proof}
    Let $\mathfrak{g} \in G \quotient G_0$. Furthermore, let $f \in \boundeds(M)$. Moreover, let $m \in M$. Because $G_0$ is finite, according to Lemma~\ref{lem:liberation-preimage}, we have $\abs{(\blank \rightsemiaction \mathfrak{g})^{-1}(m)} \leq \abs{G_0} < \infty$. Hence, the sum in the definition of $\funcamact$ is finite. Furthermore,
    \begin{align*}
      \abs{(f \funcamact \mathfrak{g})(m)}
      &\leq \sum_{m' \in (\blank \rightsemiaction \mathfrak{g})^{-1}(m)} \abs{f(m')}\\
      &\leq \parens*{\sum_{m' \in (\blank \rightsemiaction \mathfrak{g})^{-1}(m)} 1} \cdot \norm{f}_\infty\\
      &=    \abs{(\blank \rightsemiaction \mathfrak{g})^{-1}(m)} \cdot \norm{f}_\infty\\
      &\leq \abs{G_0} \cdot \norm{f}_\infty.
    \end{align*} 
    Therefore, $f \funcamact \mathfrak{g} \in \boundeds(M)$, $\norm{f \funcamact \mathfrak{g}}_\infty \leq \abs{G_0} \cdot \norm{f}_\infty$, and $\funcamact$ is well-defined. \qed
  \end{proof}

  \begin{remark}
  \label{rem:groups:funcamact}
    In the situation of Remark~\ref{rem:groups:measureamact}, we have $\funcamact \from (f, g) \mapsto [m \mapsto f(m \cdot g^{-1})]$. Hence, $\funcamact$ is the right group action of $G$ on $\R^G$ as defined in \cite[Sect.~4.3, Paragraph~5]{ceccherini-silberstein:coornaert:2010}. 
  \end{remark}

  \begin{lemma}
  \label{lem:funcamact-linear-continuous-non-negativity-preserving}
    Let $G_0$ be finite and let $\mathfrak{g}$ be an element of $G \quotient G_0$. The map $\blank \funcamact \mathfrak{g}$ is linear, continuous, and non-negativity preserving.
  \end{lemma}

  \begin{proof}
    Linearity follows from linearity of summation, continuity follows from linearity and $\norm{\blank \funcamact \mathfrak{g}}_\infty \leq \abs{G_0} \cdot \norm{\blank}_\infty$, and non-negativity preservation follows from non-negativity preservation of summation. \qed 
%
  \end{proof}

  \begin{lemma}[{\cite[Proposition~4.1.7]{ceccherini-silberstein:coornaert:2010}}] 
  \label{lem:means-are-in-dual-of-boundeds-and-have-norm-one}
    Let $\nu$ be a mean on $M$. Then, $\nu \in \boundeds(M)^*$ and $\norm{\nu}_{\boundeds(M)^*} = 1$. In particular, $\nu$ is continuous. \qed
  \end{lemma}


  \begin{definition}
    The group $G$ acts on $\boundeds(M)^*$ on the left by
    \begin{align*}
      \meanamactleft \from G \times \boundeds(M)^* &\to     \boundeds(M)^*, \mathnote{$\meanamactleft$}\\
                                         (g, \psi) &\mapsto [f \mapsto \psi(g^{-1} \funcamactleft f)],
    \end{align*}
    such that $G \meanamactleft \means(M) \subseteq \means(M)$.
  \end{definition}


  \begin{definition} 
  \label{def:kind-of-semi-action-of-quotient-on-boundeds-star}
    Let $G_0$ be finite. The quotient set $G \quotient G_0$ kind of semi-acts on $\boundeds(M)^*$ on the right by
    \begin{align*}
      \meanamact \from \boundeds(M)^* \times G \quotient G_0 &\to     \boundeds(M)^*, \mathnote{$\meanamact$}\\
                                        (\psi, \mathfrak{g}) &\mapsto [f \mapsto \psi(f \funcamact \mathfrak{g})].
    \end{align*}
  \end{definition}

  \begin{proof}
    Let $\psi \in \boundeds(M)^*$ and let $\mathfrak{g} \in G \quotient G_0$. Then, $\psi \meanamact \mathfrak{g} = \psi \after (\blank \funcamact \mathfrak{g})$. Because $\psi$ and $\blank \funcamact \mathfrak{g}$ are linear and continuous, so is $\psi \meanamact \mathfrak{g}$. \qed
  \end{proof}

  \begin{definition}
    Let $G_0$ be finite and let $\psi$ be an element of $\boundeds(M)^*$. It is called \define{$\meanamact$-invariant}\graffito{$\meanamact$-invariant} if and only if, for each element $\mathfrak{g} \in G \quotient G_0$ and each function $f \in \boundeds(M)$, we have $(\psi \meanamact \mathfrak{g})(f) = \psi(f)$.
  \end{definition}



  \begin{remark}
  \label{rem:groups:meanamact} 
    In the situation of Remark~\ref{rem:groups:funcamact}, we have $\meanamact \from (\psi, g) \mapsto [f \mapsto \psi(f \funcamact g]$. Except for $g$ not being inverted, this is the right group action of $G$ on $\boundeds(G)^*$ as defined in \cite[Sect.~4.3, Paragraph~6]{ceccherini-silberstein:coornaert:2010}. Hence, being $\meanamact$-invariant is the same as being right-invariant as defined in \cite[Sect.~4.4, Paragraph~3]{ceccherini-silberstein:coornaert:2010}. 
  \end{remark}

  %

  \begin{theorem}[{\cite[Theorem~4.1.8]{ceccherini-silberstein:coornaert:2010}}] 
  \label{thm:means-vs-measures}
    The map
    \begin{align*}
      \Phi \from \means(M) &\to     \pmeasures(M),\\
                       \nu &\mapsto [A \mapsto \nu(\unityfnc_A)],
    \end{align*}
    is bijective. \qed
  \end{theorem} 


  \begin{theorem}[{\cite[Theorem~4.2.1]{ceccherini-silberstein:coornaert:2010}}] 
  \label{thm:means-convex-and-compact-wrt-weak-star-topology}
    The set $\means(M)$ is a convex and compact subset of $\boundeds(M)^*$ equipped with the weak-$*$ topology. \qed
  \end{theorem}


  \section{Right Amenability}
  \label{sec:right-amenability}

  In Definition~\ref{def:right-amenable} we introduce the notion of right amenability using finitely additive probability measures. And in Theorem~\ref{thm:mean-characterisation-of-right-amenable} we characterise right amenability of cell spaces with finite stabilisers using means.

  \begin{definition} 
    Let $\ntuple{M, G, \leftaction}$ be a left group set. It is called \define{left amenable}\graffito{left amenable}\index{amenable!left} if and only if there is a $\measureamactleft$-invariant finitely additive probability measure on $M$.
  \end{definition}

  \begin{definition}
  \label{def:right-amenable}
    Let $\mathcal{M} = \ntuple{M, G, \leftaction}$ be a left homogeneous space. It is called \define{right amenable}\graffito{right amenable}\index{amenable!right} if and only if there is a coordinate system $\mathcal{K} = \ntuple{m_0, \family{g_{m_0, m}}_{m \in M}}$ for $\mathcal{M}$ such that there is a $\measureamact$-semi-invariant finitely additive probability measure on $M$, in which case the cell space $\mathcal{R} = \ntuple{\mathcal{M}, \mathcal{K}}$ is called \define{right amenable}\graffito{right amenable}\index{amenable!right}. 
  \end{definition}

  \begin{remark} 
    In the situation of Remark~\ref{rem:groups:measureamact}, being right amenable is the same as being amenable as defined in \cite[Definition~4.4.5]{ceccherini-silberstein:coornaert:2010}. 
  \end{remark}

  In the remainder of this section, let $\mathcal{R} = \ntuple{\ntuple{M, G, \leftaction}, \ntuple{m_0, \family{g_{m_0, m}}_{m \in M}}}$ be a cell space such that the stabiliser $G_0$ of $m_0$ under $\leftaction$ is finite. 


  \begin{lemma}
  \label{lem:indicator-function-of-A-liberation-n-is-indicator-function-of-A-acted-upon}
    Let $\mathfrak{g}$ be an element of $G \quotient G_0$ and let $A$ be a subset of $M$ such that the map $\blank \rightsemiaction \mathfrak{g}$ is injective on $A$. Then, $\unityfnc_{A \rightsemiaction \mathfrak{g}} = \unityfnc_A \funcamact \mathfrak{g}$.
  \end{lemma}

  \begin{proof}
    For each $m \in M$, because $\blank \rightsemiaction \mathfrak{g}$ is injective on $A$,
    \begin{align*}
      \unityfnc_{A \rightsemiaction \mathfrak{g}}(m)
      &= \begin{dcases*}
           1, &if $m \in A \rightsemiaction \mathfrak{g}$,\\
           0, &otherwise,
         \end{dcases*}\\
      &= \abs{\set{m' \in A \suchthat m' \rightsemiaction \mathfrak{g} = m}}\\ 
      &= \sum_{m' \in (\blank \rightsemiaction \mathfrak{g})^{-1}(m)} \unityfnc_A(m')\\
      &= (\unityfnc_A \funcamact \mathfrak{g})(m).
    \end{align*}
    In conclusion, $\unityfnc_{A \rightsemiaction \mathfrak{g}} = \unityfnc_A \funcamact \mathfrak{g}$. \qed
  \end{proof}



  \begin{lemma}[{\cite[Lemma~4.1.9]{ceccherini-silberstein:coornaert:2010}}] 
  \label{lem:simples-dense-in-boundeds}
    The vector space
    \begin{equation*}
      \simples(M) = \set{f \from M \to \R \suchthat f(M) \text{ is finite}}\ (= \vecspan\set{\unityfnc_A \suchthat A \subseteq M}) \mathnote{$\simples(M)$}
    \end{equation*}
    is dense in the Banach space $\ntuple{\boundeds(M), \norm{\blank}_\infty}$. \qed
  \end{lemma}

  \begin{lemma} 
  \label{lem:mean-invariant-if-and-only-if-semi-invariant}
    Let $\psi$ be an element of $\boundeds(M)^*$ such that, for each element $\mathfrak{g} \in G \quotient G_0$ and each subset $A$ of $M$ such that the map $\blank \rightsemiaction \mathfrak{g}$ is injective on $A$, we have $(\psi \meanamact \mathfrak{g})(\unityfnc_A) = \psi(\unityfnc_A)$. The map $\psi$ is $\meanamact$-invariant.
  \end{lemma}

  \begin{proof}
    Let $\mathfrak{g} \in G \quotient G_0$.

    First, let $A \subseteq M$. Moreover, let $m \in M$. According to Lemma~\ref{lem:liberation-preimage}, we have $k_m = \abs{(\blank \rightsemiaction \mathfrak{g})^{-1}(m)} \leq \abs{G_0}$. Hence, there are pairwise distinct $m_{m,1}$, $m_{m,2}$, $\dotsc$, $m_{m,k_m} \in M$ such that $(\blank \rightsemiaction \mathfrak{g})^{-1}(m) = \set{m_{m,1}, m_{m,2}, \dotsc, m_{m,k_m}}$. For each $i \in \set{1, 2, \dotsc, \abs{G_0}}$, put
    \begin{equation*} 
      A_i = \set{m_{m,i} \suchthat m \in M, k_m \geq i} \cap A.
    \end{equation*} 
    Because, for each $m \in M$ and each $m' \in M$ such that $m \neq m'$, we have $(\blank \rightsemiaction \mathfrak{g})^{-1}(m) \cap (\blank \rightsemiaction \mathfrak{g})^{-1}(m') = \emptyset$, the sets $A_1$, $A_2$, $\dotsc$, $A_{\abs{G_0}}$ are pairwise disjoint and the map $\blank \rightsemiaction \mathfrak{g}$ is injective on each of these sets. Moreover, because $\bigcup_{m \in M} (\blank \rightsemiaction \mathfrak{g})^{-1}(m) = M$, we have $\bigcup_{i = 1}^{\abs{G_0}} A_i = A$. Therefore, $\unityfnc_A = \sum_{i = 1}^{\abs{G_0}} \unityfnc_{A_i}$. Thus, because $\psi \meanamact \mathfrak{g}$ and $\psi$ are linear,
    \begin{equation*}
      (\psi \meanamact \mathfrak{g})(\unityfnc_A)
      = (\psi \meanamact \mathfrak{g})\parens*{\sum_{i = 1}^{\abs{G_0}} \unityfnc_{A_i}}
      = \sum_{i = 1}^{\abs{G_0}} (\psi \meanamact \mathfrak{g})(\unityfnc_{A_i})
      = \sum_{i = 1}^{\abs{G_0}} \psi(\unityfnc_{A_i})
      = \psi(\unityfnc_A).
    \end{equation*}
    Therefore, $\psi \meanamact \mathfrak{g} = \psi$ on the set of indicator functions.
%
    Thus, because the indicator functions span $\simples(M)$, and $\psi \meanamact \mathfrak{g}$ and $\psi$ are linear, $\psi \meanamact \mathfrak{g} = \psi$ on $\simples(M)$. Hence, because $\simples(M)$ is dense in $\boundeds(M)$, and $\psi \meanamact \mathfrak{g}$ and $\psi$ are continuous, $\psi \meanamact \mathfrak{g} = \psi$ on $\boundeds(M)$.
%
%
    In conclusion, $\psi$ is $\meanamact$-invariant. \qed
  \end{proof}

  \begin{theorem}
  \label{thm:mean-characterisation-of-right-amenable}
    The cell space $\mathcal{R}$ is right amenable if and only if there is a $\meanamact$-invariant mean on $M$.
  \end{theorem}

  \begin{proof} 
    Let $\Phi$ be the map in Theorem~\ref{thm:means-vs-measures}.

    First, let $\mathcal{R}$ be right amenable. Then, there is $\measureamact$-semi-invariant finitely additive probability measure $\mu$ on $M$. Put $\nu = \Phi^{-1}(\mu)$. Then, for each $\mathfrak{g} \in G \quotient G_0$ and each $A \subseteq M$ such that $\blank \rightsemiaction \mathfrak{g}$ is injective on $A$, according to Lemma~\ref{lem:indicator-function-of-A-liberation-n-is-indicator-function-of-A-acted-upon},
    \begin{equation*}
      (\nu \meanamact \mathfrak{g})(\unityfnc_A)
      = \nu(\unityfnc_A \funcamact \mathfrak{g})
      = \nu(\unityfnc_{A \rightsemiaction \mathfrak{g}})
      = \mu(A \rightsemiaction \mathfrak{g}) 
      = \mu(A)\\
      = \nu(\unityfnc_A).
    \end{equation*}
    Thus, according to Lemma~\ref{lem:mean-invariant-if-and-only-if-semi-invariant}, the mean $\nu$ is $\meanamact$-invariant.

    Secondly, let there be a $\meanamact$-invariant mean $\nu$ on $M$. Put $\mu = \Phi(\nu)$.
    Then, for each $\mathfrak{g} \in G \quotient G_0$ and each $A \subseteq M$ such that $\blank \rightsemiaction \mathfrak{g}$ is injective on $A$, according to Lemma~\ref{lem:indicator-function-of-A-liberation-n-is-indicator-function-of-A-acted-upon},
    \begin{align*}
      (\mu \measureamact \mathfrak{g})(A)
      = \mu(A \rightsemiaction \mathfrak{g})
      = \nu(\unityfnc_{A \rightsemiaction \mathfrak{g}})
      = \nu(\unityfnc_A \funcamact \mathfrak{g})
      = \nu(\unityfnc_A)
      = \mu(A).
    \end{align*}
    Hence, $\mu$ is $\measureamact$-semi-invariant. \qed
  \end{proof}

  \section{Right Følner Nets}
  \label{sec:right-folner-nets}

  In this section, let $\mathcal{R} = \ntuple{\ntuple{M, G, \leftaction}, \ntuple{m_0, \family{g_{m_0, m}}_{m \in M}}}$ be a cell space.

  \begin{definition} 
  \label{def:right-folner-net}
    Let $\net{F_i}_{i \in I}$ be a net in $\set{F \subseteq M \suchthat F \neq \emptyset, F \text{ finite}}$ indexed by $(I, \leq)$. It is called \define{right Følner net in $\mathcal{R}$ indexed by $(I, \leq)$}\graffito{right Følner net $\net{F_i}_{i \in I}$ in $\mathcal{R}$ indexed by $(I, \leq)$}\index{Følner net in $\mathcal{R}$ indexed by $(I, \leq)$!right} if and only if 
    \begin{equation*}
      \ForEach \mathfrak{g} \in G \quotient G_0 \Holds \lim_{i \in I} \frac{\abs{F_i \smallsetminus (\blank \rightsemiaction \mathfrak{g})^{-1}(F_i)}}{\abs{F_i}} = 0.
    \end{equation*} 
  \end{definition}

  \begin{remark} 
    In the situation of Remark~\ref{rem:groups:measureamact}, for each element $g \in G$ and each index $i \in I$, we have $(\blank \rightsemiaction g)^{-1}(F_i) = F_i \cdot g^{-1}$. Hence, right Følner nets in $\mathcal{R}$ are exactly right Følner nets for $G$ as defined in \cite[First paragraph after Definition~4.7.2]{ceccherini-silberstein:coornaert:2010}. 
  \end{remark}


  \begin{lemma} 
  \label{lem:characterisation-of-zero-convergent-net-that-depends-on-a-parameter}
    Let $V$ be a set, let $W$ be a set, and let $\Psi$ be a map from $V \times W$ to $\R$. There is a net $\net{v_i}_{i \in I}$ in $V$ indexed by $(I, \leq)$ such that
    \begin{equation} 
    \label{eq:characterisation-of-zero-convergent-net-that-depends-on-a-parameter:zero-convergence}
      \ForEach w \in W \Holds \lim_{i \in I} \Psi(v_i, w) = 0,
    \end{equation}
    if and only if, for each finite subset $Q$ of $W$ and each positive real number $\varepsilon \in \R_{> 0}$, there is an element $v \in V$ such that
    \begin{equation}
    \label{eq:characterisation-of-zero-convergent-net-that-depends-on-a-parameter:inequality}
      \ForEach q \in Q \Holds \Psi(v, q) < \varepsilon.
    \end{equation}
  \end{lemma}

  \begin{proof}
    First, let there be a net $\net{v_i}_{i \in I}$ in $V$ indexed by $(I, \leq)$ such that \eqref{eq:characterisation-of-zero-convergent-net-that-depends-on-a-parameter:zero-convergence} holds. Furthermore, let $Q \subseteq W$ be finite and let $\varepsilon \in \R_{> 0}$. Because \eqref{eq:characterisation-of-zero-convergent-net-that-depends-on-a-parameter:zero-convergence} holds, for each $q \in Q$, there is an $i_q \in I$ such that,
    \begin{equation*}
      \ForEach i \in I \Holds (i \geq i_q \implies \Psi(v_i, q) < \varepsilon).
    \end{equation*}
    Because $(I, \leq)$ is a directed set and $Q$ is finite, there is an $i \in I$ such that, for each $q \in Q$, we have $i \geq i_q$. Put $v = v_i$. Then, \eqref{eq:characterisation-of-zero-convergent-net-that-depends-on-a-parameter:inequality} holds.

    Secondly, for each finite $Q \subseteq W$ and each $\varepsilon \in \R_{> 0}$, let there be a $v \in V$ such that \eqref{eq:characterisation-of-zero-convergent-net-that-depends-on-a-parameter:inequality} holds. Furthermore, let
    \begin{equation*}
      I = \set{Q \subseteq W \suchthat Q \text{ is finite}} \times \R_{> 0}
    \end{equation*}
    and let $\leq$ be the preorder on $I$ given by
    \begin{equation*}
      \ForEach (Q, \varepsilon) \in I \ForEach (Q', \varepsilon') \in I \Holds
          (Q, \varepsilon) \leq (Q', \varepsilon') \iff Q \subseteq Q' \land \varepsilon \geq \varepsilon'.
    \end{equation*}
    For each $(Q, \varepsilon) \in I$ and each $(Q', \varepsilon') \in I$, the element $(Q \cup Q', \min(\varepsilon, \varepsilon'))$ of $I$ is an upper bound of $(Q, \varepsilon)$ and of $(Q', \varepsilon')$. Hence, $(I, \leq)$ is a directed set.

    By precondition, for each $i = (Q, \varepsilon) \in I$, there is a $v_i \in V$ such that
    \begin{equation*}
      \ForEach q \in Q \Holds \Psi(v_i, q) < \varepsilon.
    \end{equation*}
    Let $w \in W$ and let $\varepsilon_0 \in \R_{> 0}$. Put $i_0 = (\set{w}, \varepsilon_0)$. For each $i = (Q, \varepsilon) \in I$ with $i \geq i_0$, we have $w \in Q$ and $\varepsilon \leq \varepsilon_0$. Hence,
    \begin{equation*}
      \ForEach i \in I \Holds (i \geq i_0 \implies \Psi(v_i, w) < \varepsilon_0).
    \end{equation*}
    Therefore, $\net{v_i}_{i \in I}$ is a net in $V$ indexed by $(I, \leq)$ such that \eqref{eq:characterisation-of-zero-convergent-net-that-depends-on-a-parameter:zero-convergence} holds. \qed
  \end{proof}

  \begin{lemma} 
  \label{lem:epsilon-characterisation-of-folner-net} 
    There is a right Følner net in $\mathcal{R}$ if and only if, for each finite subset $E$ of $G \quotient G_0$ and each positive real number $\varepsilon \in \R_{> 0}$, there is a non-empty and finite subset $F$ of $M$ such that
    \begin{equation*}
      \ForEach e \in E \Holds \frac{\abs{F \smallsetminus (\blank \rightsemiaction e)^{-1}(F)}}{\abs{F}} < \varepsilon.
    \end{equation*}
  \end{lemma}

  \begin{proof}
    This is a direct consequence of Lemma~\ref{lem:characterisation-of-zero-convergent-net-that-depends-on-a-parameter} with
    \begin{align*}
      \Psi \from \set{F \subseteq M \suchthat F \neq \emptyset, F \text{ finite}} \times G \quotient G_0 &\to     \R,\\
                                                                                       (F, \mathfrak{g}) &\mapsto \frac{\abs{F \smallsetminus (\blank \rightsemiaction \mathfrak{g})^{-1}(F)}}{\abs{F}}. \tag*{\qed}
    \end{align*}
  \end{proof}

%

  \begin{lemma}
  \label{lem:rightsemiaction-can-be-undone}
    Let $m$ be an element of $M$, and let $\mathfrak{g}$ be an element of $G \quotient G_0$. There is an element $g \in \mathfrak{g}$ such that
    \begin{equation*}
      \ForEach \mathfrak{g}' \in G \quotient G_0 \Holds (m \rightsemiaction \mathfrak{g}) \rightsemiaction \mathfrak{g}' = m \rightsemiaction g \cdot \mathfrak{g}'.
    \end{equation*}
  \end{lemma}

  \begin{proof}
    There is a $g \in G$ such that $g G_0 = \mathfrak{g}$. Moreover, because $\rightsemiaction$ is a semi-action with defect $G_0$, there is a $g_0 \in G_0$ such that
    \begin{equation*} 
      \ForEach \mathfrak{g}' \in G \quotient G_0 \Holds (m \rightsemiaction g G_0) \rightsemiaction \mathfrak{g}' = m \rightsemiaction g \cdot (g_0^{-1} \cdot \mathfrak{g}').
    \end{equation*}
    Because $g \cdot (g_0^{-1} \cdot \mathfrak{g}') = g g_0^{-1} \cdot \mathfrak{g}'$ and $g g_0^{-1} \in \mathfrak{g}$, the statement holds. \qed
  \end{proof}

  \begin{lemma}
  \label{lem:liberation-by-n-yields-element-in-set-setminus-bigcup-liberation-by-inverse-times-n-prime}
    Let $A$ and $A'$ be two subsets of $M$, and let $\mathfrak{g}$ and $\mathfrak{g}'$ be two elements of $G \quotient G_0$. Then, for each element $m \in (\blank \rightsemiaction \mathfrak{g})^{-1}(A) \smallsetminus (\blank \rightsemiaction \mathfrak{g}')^{-1}(A')$,
    \begin{gather*}
      m \rightsemiaction \mathfrak{g} \in \bigcup_{g \in \mathfrak{g}} A \smallsetminus (\blank \rightsemiaction g^{-1} \cdot \mathfrak{g}')^{-1}(A'),\\
      m \rightsemiaction \mathfrak{g}' \in \bigcup_{g' \in \mathfrak{g}'} (\blank \rightsemiaction (g')^{-1} \cdot \mathfrak{g})^{-1}(A) \smallsetminus A'.
    \end{gather*}
  \end{lemma}

  \begin{proof}[Lemma~\ref{lem:liberation-by-n-yields-element-in-set-setminus-bigcup-liberation-by-inverse-times-n-prime}]
    Let $m \in (\blank \rightsemiaction \mathfrak{g})^{-1}(A) \smallsetminus (\blank \rightsemiaction \mathfrak{g}')^{-1}(A')$. Then, $m \rightsemiaction \mathfrak{g} \in A$ and $m \rightsemiaction \mathfrak{g}' \notin A'$. According to Lemma~\ref{lem:rightsemiaction-can-be-undone}, there is a $g \in \mathfrak{g}$ and a $g' \in \mathfrak{g}'$ such that $(m \rightsemiaction \mathfrak{g}) \rightsemiaction g^{-1} \cdot \mathfrak{g}' = m \rightsemiaction \mathfrak{g}' \notin A'$ and $(m \rightsemiaction \mathfrak{g}') \rightsemiaction (g')^{-1} \cdot \mathfrak{g} = m \rightsemiaction \mathfrak{g} \in A$. Hence, $m \rightsemiaction \mathfrak{g} \notin (\blank \rightsemiaction g^{-1} \cdot \mathfrak{g}')^{-1}(A')$ and $m \rightsemiaction \mathfrak{g}' \in (\blank \rightsemiaction (g')^{-1} \cdot \mathfrak{g})^{-1}(A)$. Therefore, $m \rightsemiaction \mathfrak{g} \in A \smallsetminus (\blank \rightsemiaction g^{-1} \cdot \mathfrak{g}')^{-1}(A')$ and $m \rightsemiaction \mathfrak{g}' \in (\blank \rightsemiaction (g')^{-1} \cdot \mathfrak{g})^{-1}(A) \smallsetminus A'$. In conclusion, $m \rightsemiaction \mathfrak{g} \in \bigcup_{g \in \mathfrak{g}} A \smallsetminus (\blank \rightsemiaction g^{-1} \cdot \mathfrak{g}')^{-1}(A')$ and $m \rightsemiaction \mathfrak{g}' \in \bigcup_{g' \in \mathfrak{g}'} (\blank \rightsemiaction (g')^{-1} \cdot \mathfrak{g})^{-1}(A) \smallsetminus A'$. \qed
  \end{proof}

  \begin{lemma}
  \label{lem:cardinality-of-inverse-image-of-liberation-minus-the-same-less-than-or-equal-to-whatever}
    Let $G_0$ be finite, let $F$ and $F'$ be two finite subsets of $M$, and let $\mathfrak{g}$ and $\mathfrak{g}'$ be two elements of $G \quotient G_0$. Then,
    \begin{equation*}
      \abs{(\blank \rightsemiaction \mathfrak{g})^{-1}(F) \smallsetminus (\blank \rightsemiaction \mathfrak{g}')^{-1}(F')}
      \leq
      \begin{dcases*} 
        \abs{G_0}^2 \cdot \max_{g \in \mathfrak{g}} \abs{F \smallsetminus (\blank \rightsemiaction g^{-1} \cdot \mathfrak{g}')^{-1}(F')},\\
        \abs{G_0}^2 \cdot \max_{g' \in \mathfrak{g}'} \abs{(\blank \rightsemiaction (g')^{-1} \cdot \mathfrak{g})^{-1}(F) \smallsetminus F'}.
      \end{dcases*}
    \end{equation*}
  \end{lemma}

  \begin{proof}
    Put $A = (\blank \rightsemiaction \mathfrak{g})^{-1}(F) \smallsetminus (\blank \rightsemiaction \mathfrak{g}')^{-1}(F')$. For each $g \in \mathfrak{g}$, put $B_g = F \smallsetminus (\blank \rightsemiaction g^{-1} \cdot \mathfrak{g}')^{-1}(F')$. For each $g' \in \mathfrak{g}'$, put $B_{g'}' = (\blank \rightsemiaction (g')^{-1} \cdot \mathfrak{g})^{-1}(F) \smallsetminus F'$.

    According to Lemma~\ref{lem:liberation-by-n-yields-element-in-set-setminus-bigcup-liberation-by-inverse-times-n-prime}, the restrictions $(\blank \rightsemiaction \mathfrak{g})\restriction_{A \to \bigcup_{g \in \mathfrak{g}} B_g}$ and $(\blank \rightsemiaction \mathfrak{g}')\restriction_{A \to \bigcup_{g' \in \mathfrak{g}'} B_{g'}'}$ are well-defined. Moreover, for each $m \in M$, according to Lemma~\ref{lem:liberation-preimage}, we have $\abs{(\blank \rightsemiaction \mathfrak{g})^{-1}(m)} \leq \abs{G_0}$ and $\abs{(\blank \rightsemiaction \mathfrak{g}')^{-1}(m)} \leq \abs{G_0}$. Therefore, because $\abs{\mathfrak{g}} = \abs{G_0}$, 
    \begin{equation*}
      \abs{A}
      \leq \abs{G_0} \cdot \abs{\bigcup_{g \in \mathfrak{g}} B_g}
      \leq \abs{G_0} \cdot \sum_{g \in \mathfrak{g}} \abs{B_g}
      \leq \abs{G_0}^2 \cdot \max_{g \in \mathfrak{g}} \abs{B_g}
    \end{equation*}
    and analogously
    \begin{equation*}
      \abs{A} \leq \abs{G_0}^2 \cdot \max_{g' \in \mathfrak{g}'} \abs{B_{g'}'}. \tag*{\qed}
    \end{equation*}
  \end{proof}

  \begin{lemma} 
  \label{lem:reversed-sets-characterisation-of-folner-nets}
    Let $G_0$ be finite and let $\net{F_i}_{i \in I}$ be a net in $\set{F \subseteq M \suchthat F \neq \emptyset, F \text{ finite}}$ indexed by $(I, \leq)$. The net $\net{F_i}_{i \in I}$ is a right Følner net in $\mathcal{R}$ if and only if
    \begin{equation}
    \label{eq:reversed-sets-characterisation-of-folner-nets:condition}
      \ForEach \mathfrak{g} \in G \quotient G_0 \Holds \lim_{i \in I} \frac{\abs{(\blank \rightsemiaction \mathfrak{g})^{-1}(F_i) \smallsetminus F_i}}{\abs{F_i}} = 0.
    \end{equation}
  \end{lemma}

  \begin{proof}
    Let $\mathfrak{g} \in G \quotient G_0$. Furthermore, let $i \in I$. Because $F_i = (\blank \rightsemiaction G_0)^{-1}(F_i)$, according to Lemma~\ref{lem:cardinality-of-inverse-image-of-liberation-minus-the-same-less-than-or-equal-to-whatever},
    \begin{equation*}
      \abs{(\blank \rightsemiaction \mathfrak{g})^{-1}(F_i) \smallsetminus F_i}
      \leq \abs{G_0}^2 \cdot \max_{g \in \mathfrak{g}} \abs{F_i \smallsetminus (\blank \rightsemiaction g^{-1} G_0)^{-1}(F_i)}
    \end{equation*}
    and
    \begin{equation*}
      \abs{F_i \smallsetminus (\blank \rightsemiaction \mathfrak{g})^{-1}(F_i)}
      \leq \abs{G_0}^2 \cdot \max_{g \in \mathfrak{g}} \abs{(\blank \rightsemiaction g^{-1} G_0)^{-1}(F_i) \smallsetminus F_i}.
    \end{equation*}
    Moreover, $\abs{\mathfrak{g}} = \abs{G_0} < \infty$. Therefore, if $\net{F_i}_{i \in I}$ is a right Følner net in $\mathcal{R}$, then
    \begin{equation*}
      \lim_{i \in I} \frac{\abs{(\blank \rightsemiaction \mathfrak{g})^{-1}(F_i) \smallsetminus F_i}}{\abs{F_i}} = 0;
    \end{equation*}
    and, if \eqref{eq:reversed-sets-characterisation-of-folner-nets:condition} holds, then
    \begin{equation*}
      \lim_{i \in I} \frac{\abs{F_i \smallsetminus (\blank \rightsemiaction \mathfrak{g})^{-1}(F_i)}}{\abs{F_i}} = 0.
    \end{equation*}
    In conclusion, $\net{F_i}_{i \in I}$ is a right Følner net in $\mathcal{R}$ if and only if \eqref{eq:reversed-sets-characterisation-of-folner-nets:condition} holds.
  %
  \end{proof}

  \begin{lemma} 
    Let $G_0$ be finite. There is a right Følner net in $\mathcal{R}$ if and only if, for each finite subset $E$ of $G \quotient G_0$ and each positive real number $\varepsilon \in \R_{> 0}$, there is a non-empty and finite subset $F$ of $M$ such that
    \begin{equation*}
      \ForEach e \in E \Holds \frac{\abs{(\blank \rightsemiaction e)^{-1}(F) \smallsetminus F}}{\abs{F}} < \varepsilon.
    \end{equation*}
  \end{lemma}

  \begin{proof}
    This is a direct consequence of Lemma~\ref{lem:reversed-sets-characterisation-of-folner-nets} and Lemma~\ref{lem:characterisation-of-zero-convergent-net-that-depends-on-a-parameter} with
    \begin{align*}
      \Psi \from \set{F \subseteq M \suchthat F \neq \emptyset, F \text{ finite}} \times G \quotient G_0 &\to     \R,\\
                                                                                       (F, \mathfrak{g}) &\mapsto \frac{\abs{(\blank \rightsemiaction \mathfrak{g})^{-1}(F) \smallsetminus F}}{\abs{F}}. \tag*{\qed}
    \end{align*}
  \end{proof}

  \section{Right Paradoxical Decompositions}
  \label{sec:right-paradoxical-decomposition}

  In this section, let $\mathcal{R} = \ntuple{\ntuple{M, G, \leftaction}, \ntuple{m_0, \family{g_{m_0, m}}_{m \in M}}}$ be a cell space.

  \begin{definition} 
            Let $A$ and $A'$ be two sets. The set $A \cup A'$ is denoted by $A \cupdot A'$ if and only if the sets $A$ and $A'$ are disjoint.
  \end{definition}

  \begin{definition}
  \label{def:right-paradoxical-decomposition}
    Let $E$ be a finite subset of $G \quotient G_0$, and let $\family{A_e}_{e \in E}$ and $\family{B_e}_{e \in E}$ be two families of subsets of $M$ indexed by $E$ such that, for each index $e \in E$, the map $\blank \rightsemiaction e$ is injective on $A_e$ and on $B_e$, and
    \begin{equation*}
      M = \bigcupdot_{e \in E} A_e
        = \bigcupdot_{e \in E} B_e
        = \parens*{\bigcupdot_{e \in E} A_e \rightsemiaction e} \cupdot \parens*{\bigcupdot_{e \in E} B_e \rightsemiaction e}.
    \end{equation*}
    The triple $\ntuple{N, \family{A_e}_{e \in E}, \family{B_e}_{e \in E}}$ is called \define{right paradoxical decomposition of $\mathcal{R}$}\graffito{right paradoxical decomposition $\ntuple{N, \family{A_e}_{e \in E}, \family{B_e}_{e \in E}}$ of $\mathcal{R}$}\index{paradoxical decomposition of $\mathcal{R}$!right}.
  \end{definition}

  \begin{remark} 
    In the situation of Remark~\ref{rem:groups:measureamact}, for each element $g \in G$, the map $\blank \rightsemiaction g$ is injective. Hence, right paradoxical decompositions of $\mathcal{R}$ are the same as right paradoxical decompositions of $G$ as defined in \cite[Definition~4.8.1]{ceccherini-silberstein:coornaert:2010}. 
  \end{remark}

  \begin{lemma}
  \label{lem:functional-right-paradoxical-decomposition}
    Let $G_0$ be finite and let $\ntuple{N, \family{A_e}_{e \in E}, \family{B_e}_{e \in E}}$ be a right paradoxical decomposition of $\mathcal{R}$. Then,
    \begin{equation*}
      \unityfnc_M = \sum_{e \in E} \unityfnc_{A_e}
                  = \sum_{e \in E} \unityfnc_{B_e}
                  = \sum_{e \in E} (\unityfnc_{A_e} \funcamact e) + \sum_{e \in E} (\unityfnc_{B_e} \funcamact e).
    \end{equation*}
  \end{lemma}

  \begin{proof}
    This is a direct consequence of Definition~\ref{def:right-paradoxical-decomposition} and Lemma~\ref{lem:indicator-function-of-A-liberation-n-is-indicator-function-of-A-acted-upon}.
  \end{proof}

  \section{Tarski's and Følner's Theorem}
  \label{sec:tarski-folner-theorem}

  In this section, let $\mathcal{R} = \ntuple{\ntuple{M, G, \leftaction}, \ntuple{m_0, \family{g_{m_0, m}}_{m \in M}}}$ be a cell space such that the stabiliser $G_0$ of $m_0$ under $\leftaction$ is finite.

  \begin{lemma} 
  \label{lem:meanamact-is-continuous}
    Let $\mathfrak{g}$ be an element of $G \quotient G_0$. The map $\blank \meanamact \mathfrak{g}$ is continuous, where $\boundeds(M)^*$ is equipped with the weak-$*$ topology.
  \end{lemma}

  \begin{proof}
    For each $f \in \boundeds(M)$, let
      $\ev_f \from \boundeds(M)^* \to     \R,
                             \psi \mapsto \psi(f)$.
    Furthermore, let $f \in \boundeds(M)$. Then, for each $\psi \in \boundeds(M)^*$,
    \begin{equation*}
      (\ev_f \circ (\blank \meanamact \mathfrak{g}))(\psi)
      = \ev_f(\psi \meanamact \mathfrak{g})
      = (\psi \meanamact \mathfrak{g})(f)
      = \psi(f \funcamact \mathfrak{g})
      = \ev_{f \funcamact \mathfrak{g}}(\psi).
    \end{equation*}
    Thus, $\ev_f \circ (\blank \meanamact \mathfrak{g}) = \ev_{f \funcamact \mathfrak{g}}$. Hence, because $\ev_{f \funcamact \mathfrak{g}}$ is continuous, so is $\ev_f \circ (\blank \meanamact \mathfrak{g})$. Therefore, the map $\blank \meanamact \mathfrak{g}$ is continuous. \qed 
  \end{proof}

  \begin{lemma} 
  \label{lem:convergent-net-sufficient-for-amenability}
    Let $\net{\nu_i}_{i \in I}$ be a net in $\means(M)$ such that, for each element $\mathfrak{g} \in G \quotient G_0$, the net $\net{(\nu_i \meanamact \mathfrak{g}) - \nu_i}_{i \in I}$ converges to $\nullityfnc$ in $\boundeds(M)^*$ equipped with the weak-$*$ topology. The cell space $\mathcal{R}$ is right amenable.
  \end{lemma}

  \begin{proof} 
    Let $\mathfrak{g} \in G \quotient G_0$. According to Theorem~\ref{thm:means-convex-and-compact-wrt-weak-star-topology}, the set $\means(M)$ is compact in $\boundeds(M)^*$ equipped with the weak-$*$ topology. Hence, there is a subnet $\net{\nu_{i_j}}_{j \in J}$ of $\net{\nu_i}_{i \in I}$ that converges to a $\nu \in \means(M)$. Because, according to Lemma~\ref{lem:meanamact-is-continuous}, the map $\blank \meanamact \mathfrak{g}$ is continuous, the net $\net{(\nu_{i_j} \meanamact \mathfrak{g}) - \nu_{i_j}}_{j \in J}$ converges to $(\nu \meanamact \mathfrak{g}) - \nu$ in $\boundeds(M)^*$. Because it is a subnet of $\net{(\nu_i \meanamact \mathfrak{g}) - \nu_i}_{i \in I}$, it also converges to $\nullityfnc$ in $\boundeds(M)^*$. Because the space $\boundeds(M)^*$ is Hausdorff, we have $(\nu \meanamact \mathfrak{g}) - \nu = \nullityfnc$ and hence $\nu \meanamact \mathfrak{g} = \nu$. Altogether, $\nu$ is a $\meanamact$-invariant mean. In conclusion, according to Theorem~\ref{thm:mean-characterisation-of-right-amenable}, the cell space $\mathcal{R}$ is right amenable. \qed
  \end{proof} 

  \begin{lemma} 
  \label{lem:double-liberation-with-sets-to-one-liberation}
    Let $m$ be an element of $M$, and let $E$ and $E'$ be two subsets of $G \quotient G_0$. There is a subset $E''$ of $G \quotient G_0$ such that $(m \rightsemiaction E) \rightsemiaction E' = m \rightsemiaction E''$; if $G_0 \in E \cap E'$, then $G_0 \in E''$; if $E$ and $E'$ are finite, then $\abs{E''} \leq \abs{E} \cdot \abs{E'}$; and if $G_0 \cdot E' \subseteq E'$, then
      $E'' = \set{g \cdot e' \suchthat e \in E, e' \in E', g \in e}$. 
  \end{lemma}

  \begin{proof}
    For each $e \in E$, according to Lemma~\ref{lem:rightsemiaction-can-be-undone}, there is a $g_e \in e$ such that 
    \begin{equation*}
      \ForEach \mathfrak{g} \in G \quotient G_0 \Holds (m \rightsemiaction e) \rightsemiaction \mathfrak{g} = m \rightsemiaction g_e \cdot \mathfrak{g}. 
    \end{equation*}
    Put $E'' = \set{g_e \cdot e' \suchthat e \in E, e' \in E'}$. Then, $(m \rightsemiaction E) \rightsemiaction E' = m \rightsemiaction E''$. Moreover, if $G_0 \in E \cap E'$, then $G_0 = g_{G_0} \cdot G_0 \in E''$; if $E$ and $E'$ are finite, then $\abs{E''} \leq \abs{E} \cdot \abs{E'}$; and if $G_0 \cdot E' \subseteq E'$, then $E''$ is as stated. \qed 
  \end{proof}

  \begin{main-theorem} 
  \label{thm:tarski-folner}
    Let $\mathcal{R} = \ntuple{\ntuple{M, G, \leftaction}, \ntuple{m_0, \family{g_{m_0, m}}_{m \in M}}}$ be a cell space such that the stabiliser $G_0$ of $m_0$ under $\leftaction$ is finite. The following statements are equivalent:
    \begin{enumerate}
      \item \label{it:tarski-folner:not-right-amenable}
            The cell space $\mathcal{R}$ is not right amenable;
      \item \label{it:tarski-folner:no-folner-net}
            There is no right Følner net in $\mathcal{R}$;
      \item \label{it:tarski-folner:finite-N}
            There is a finite subset $E$ of $G \quotient G_0$ such that $G_0 \in E$ and, for each finite subset $F$ of $M$, we have $\abs{F \rightsemiaction E} \geq 2 \abs{F}$;
      \item \label{it:tarski-folner:two-to-one-surjective-map}
            There is a $2$-to-$1$ surjective map $\phi \from M \to M$ and there is a finite subset $E$ of $G \quotient G_0$ such that
            \begin{equation*}
              \ForEach m \in M \Exists e \in E \SuchThat \phi(m) \rightsemiaction e = m;
            \end{equation*}
      \item \label{it:tarski-folner:no-paradoxical-decomposition}
            There is a right paradoxical decomposition of $\mathcal{R}$.
    \end{enumerate}
  \end{main-theorem}

  \begin{proof}
    \begin{description}
      \item[\ref{it:tarski-folner:not-right-amenable} implies \ref{it:tarski-folner:no-folner-net}.] 
            Let there be a right Følner net $\net{F_i}_{i \in I}$ in $\mathcal{R}$. Furthermore, let $i \in I$. Put
            \begin{align*}
              \nu_i \from \boundeds(M) &\to     \R,\\
                                     f &\mapsto \frac{1}{\abs{F_i}} \sum_{m \in F_i} f(m).
            \end{align*}
            Then, $\nu_i \in \means(M)$. 
            Moreover, let $\mathfrak{g} \in G \quotient G_0$ and let $f \in \boundeds(M)$. Then,
            \begin{align*}
              (\nu_i \meanamact \mathfrak{g})(f)
              &= \nu_i(f \funcamact \mathfrak{g})\\
              &= \frac{1}{\abs{F_i}} \sum_{m \in F_i} (f \funcamact \mathfrak{g})(m)\\
              &= \frac{1}{\abs{F_i}} \sum_{m \in F_i} \sum_{m' \in (\blank \rightsemiaction \mathfrak{g})^{-1}(m)} f(m')\\
              &= \frac{1}{\abs{F_i}} \sum_{m \in (\blank \rightsemiaction \mathfrak{g})^{-1}(F_i)} f(m).
            \end{align*}
            Hence,
            \begin{equation*}
              (\nu_i \meanamact \mathfrak{g} - \nu_i)(f)
              = \frac{1}{\abs{F_i}} \parens*{\sum_{m \in (\blank \rightsemiaction \mathfrak{g})^{-1}(F_i) \smallsetminus F_i} f(m)
                 - \sum_{m \in F_i \smallsetminus (\blank \rightsemiaction \mathfrak{g})^{-1}(F_i)} f(m)}.
            \end{equation*}
            Therefore,
            \begin{align*}
              \abs{(\nu_i \meanamact \mathfrak{g} - \nu_i)(f)}
              &\leq \frac{1}{\abs{F_i}} \parens*{\sum_{m \in (\blank \rightsemiaction \mathfrak{g})^{-1}(F_i) \smallsetminus F_i} \abs{f(m)}
                    + \sum_{m \in F_i \smallsetminus (\blank \rightsemiaction \mathfrak{g})^{-1}(F_i)} \abs{f(m)}}\\
              &\leq 
                    \parens*{\frac{\abs{(\blank \rightsemiaction \mathfrak{g})^{-1}(F_i) \smallsetminus F_i}}{\abs{F_i}}
                    + \frac{\abs{F_i \smallsetminus (\blank \rightsemiaction \mathfrak{g})^{-1}(F_i)}}{\abs{F_i}}} \cdot \norm{f}_\infty.
            \end{align*}
            According to Definition~\ref{def:right-folner-net} and Lemma~\ref{lem:reversed-sets-characterisation-of-folner-nets}, the nets $\net{\abs{(\blank \rightsemiaction \mathfrak{g})^{-1}(F_i) \smallsetminus F_i} / \abs{F_i}}_{i \in I}$ and $\net{\abs{F_i \smallsetminus (\blank \rightsemiaction \mathfrak{g})^{-1}(F_i)} / \abs{F_i}}_{i \in I}$ converge to $0$. Hence, so does $\net{\abs{(\nu_i \meanamact \mathfrak{g} - \nu_i)(f)}}_{i \in I}$. Thus, the net $\net{\nu_i \meanamact \mathfrak{g} - \nu_i}_{i \in I}$ converges to $\nullityfnc$ in $\boundeds(M)^*$ equipped with the weak-$*$ topology. Hence, according to Lemma~\ref{lem:convergent-net-sufficient-for-amenability}, the cell space $\mathcal{R}$ is right amenable. In conclusion, by contraposition, if $\mathcal{R}$ is not right amenable, then there is no right Følner net in $\mathcal{R}$.
      \item[\ref{it:tarski-folner:no-folner-net} implies \ref{it:tarski-folner:finite-N}.]
            Let there be no right Følner net in $\mathcal{R}$. According to Lemma~\ref{lem:epsilon-characterisation-of-folner-net}, there is a finite $E_1 \subseteq G \quotient G_0$ and an $\varepsilon \in \R_{> 0}$ such that, for each non-empty and finite $F \subseteq M$, there is an $e_F \in E_1$ such that 
            \begin{equation*}
              \frac{\abs{F \smallsetminus (\blank \rightsemiaction e_F)^{-1}(F)}}{\abs{F}} \geq \varepsilon.
            \end{equation*}
            Put $E_2 = \set{G_0} \cup E_1$.

              Let $F \subseteq M$ be non-empty and finite. Then, $F \subseteq F \cup (F \rightsemiaction E_1) = F \rightsemiaction E_2$. Thus,
              \begin{align*}
                \abs{F \rightsemiaction E_2} - \abs{F}
                &=    \abs{(F \rightsemiaction E_2) \smallsetminus F}\\
                &=    \abs{(F \rightsemiaction E_1) \smallsetminus F}\\
                &\geq \abs{(F \rightsemiaction e_F) \smallsetminus F}.
              \end{align*}
              Moreover, according to Lemma~\ref{lem:liberation-preimage}, we have $\abs{(\blank \rightsemiaction e_F)^{-1}((F \rightsemiaction e_F) \smallsetminus F)} \leq \abs{G_0} \cdot \abs{(F \rightsemiaction e_F) \smallsetminus F}$. Hence,
              \begin{equation*}
                \abs{F \rightsemiaction E_2} - \abs{F}
                \geq \frac{\abs{(\blank \rightsemiaction e_F)^{-1}((F \rightsemiaction e_F) \smallsetminus F)}}{\abs{G_0}}.
              \end{equation*}
              Therefore, because $F \smallsetminus (\blank \rightsemiaction e_F)^{-1}(F) \subseteq (\blank \rightsemiaction e_F)^{-1}((F \rightsemiaction e_F) \smallsetminus F)$, 
              \begin{align*}
                \abs{F \rightsemiaction E_2} - \abs{F}
                &\geq \frac{\abs{F \smallsetminus (\blank \rightsemiaction e_F)^{-1}(F)}}{\abs{G_0}}\\
                &\geq \frac{\varepsilon}{\abs{G_0}} \abs{F}.
              \end{align*}
              Put $\xi = 1 + \varepsilon / \abs{G_0}$. Then, $\abs{F \rightsemiaction E_2} \geq \xi \abs{F}$. Because $\varepsilon$ does not depend on $F$, neither does $\xi$. Therefore, for each non-empty and finite $F \subseteq M$, we have $\abs{F \rightsemiaction E_2} \geq \xi \abs{F}$.

            Let $F \subseteq M$ be non-empty and finite. Because $\xi > 1$, there is an $n \in \N$ such that $\xi^n \geq 2$. Hence,
            \begin{align*} 
              \abs{\underbrace{(((F \rightsemiaction E_2) \rightsemiaction \dotsb) \rightsemiaction E_2) \rightsemiaction E_2}_{\text{$n$ times}}}
              &\geq \xi \abs{((F \rightsemiaction E_2) \rightsemiaction \dotsb) \rightsemiaction E_2}\\
              &\geq \dotsb\\
              &\geq \xi^n \abs{F}\\
              &\geq 2 \abs{F}.
            \end{align*}
            Moreover, according to Lemma~\ref{lem:double-liberation-with-sets-to-one-liberation}, there is an $E \subseteq G \quotient G_0$ such that $E$ is finite, $G_0 \in E$, and $F \rightsemiaction E = (((F \rightsemiaction E_2) \rightsemiaction \dotsb) \rightsemiaction E_2) \rightsemiaction E_2$. In conclusion, $\abs{F \rightsemiaction E} \geq 2 \abs{F}$.
      \item[\ref{it:tarski-folner:finite-N} implies \ref{it:tarski-folner:two-to-one-surjective-map} (see Fig.~\ref{fig:finite-N-implies-two-to-one-surjective-map}).] 
            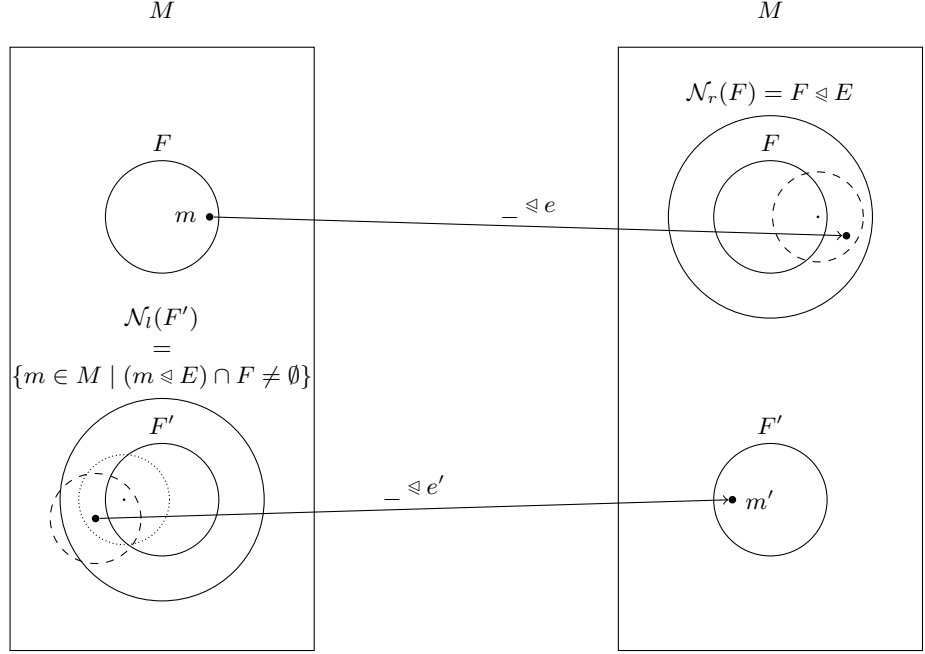
\begin{figure}[bt]
              \centering
              \begin{tikzpicture} 
                \pgfmathsetmacro\w{4} 
                \pgfmathsetmacro\cox{8} 

                \draw (0, 0) rectangle (\w, 8);
                \draw (\cox, 0) rectangle (\cox + \w, 8);

                \node (M) at (\w / 2, 8.5) {$M$};
                \node (M) at (\cox + \w / 2, 8.5) {$M$};

                \node[fill, circle, inner sep = 1pt, label = left: $m$] (m) at (\w - 1.375, 5.75) {}; 
                \node[draw, circle, inner sep = 15pt, label = above: $F$] (F) at (2, 5.75) {};
                \node[draw, circle, inner sep = 15pt, label = above: $F$] (Fx) at (\cox + 2, 5.75) {};

                \node[draw, circle, inner sep = 27pt, label = {above: $\set{m \in M \suchthat (m \rightsemiaction E) \cap F \neq \emptyset}$}] (Nl) at (2, 2) {}; 
                \node[inner sep = 51pt, label = {above: $=$}] (Nl) at (2, 2) {}; 
                \node[inner sep = 60pt, label = {above: $\mathcal{N}_l(F')$}] (Nl) at (2, 2) {}; 

                \node[fill, circle, inner sep = 0.25pt] (mxx) at (\cox + \w - 1.375, 5.75) {};
                \node[draw, circle, inner sep = 12pt, dashed] (m lib N) at (\cox + \w - 1.375, 5.75) {}; 
                \node[fill, circle, inner sep = 1pt] (m lib n) at (\cox + 3, 5.5) {}; 
                \node[draw, circle, inner sep = 27pt, label = {above: $\mathcal{N}_r(F) = F \rightsemiaction E$}] (Nr) at (\cox + 2, 5.75) {};
                \node[fill, circle, inner sep = 1pt, label = right: $m'$] (m') at (\cox + 1.5, 2) {};
                \node[draw, circle, inner sep = 15pt, label = above: $F'$] (F') at (\cox + 2, 2) {};
                \node[draw, circle, inner sep = 15pt, label = above: $F'$] (F'x) at (2, 2) {};
                \node[fill, circle, inner sep = 0.25pt] (m'xx) at (1.5, 2) {};
                \node[draw, circle, inner sep = 12pt, densely dotted] (m lib N) at (1.5, 2) {}; 
                \node[fill, circle, inner sep = 1pt] (lib n inv m') at (1.125, 1.75) {};
                \node[draw, circle, inner sep = 12pt, dashed] (lib n inv m' lib N) at (1.125, 1.75) {};

                \draw[->] (m) edge node[above] {$\blank \rightsemiaction e$} (m lib n)
                          (lib n inv m') edge node[above] {$\blank \rightsemiaction e'$} (m');
              \end{tikzpicture}
              \caption{
                Schematic representation of the set-up of the proof of Theorem~\ref{thm:tarski-folner}, Item~\ref{it:tarski-folner:finite-N} implies Item~\ref{it:tarski-folner:two-to-one-surjective-map}: Each region enclosed by one of the two rectangles is $M$; the regions enclosed by the smaller circles with solid borders are subsets $F$ and $F'$ of $M$ respectively; the regions enclosed by the circles with dashed borders are $m \rightsemiaction E$ and $m' \rightsemiaction E$ respectively; the region enclosed by the circle with dotted border is $\bigcup_{e' \in E} (\blank \rightsemiaction e')^{-1}(m')$; the regions enclosed by the larger circles are $\mathcal{N}_r(F)$ and $\mathcal{N}_l(F')$ respectively.
              }
              \label{fig:finite-N-implies-two-to-one-surjective-map}
            \end{figure}
            Let there be a finite $E \subseteq G \quotient G_0$ such that, for each finite $F \subseteq M$, we have $\abs{F \rightsemiaction E} \geq 2 \abs{F}$. Furthermore, let $\mathcal{G}$ be the bipartite graph 
            \begin{equation*}
              \ntuple{M, M, \set{(m, m') \in M \times M \suchthat \Exists e \in E \SuchThat m \rightsemiaction e = m'}}.
            \end{equation*}
            Moreover, let $F \subseteq M$ be finite. The right neighbourhood of $F$ in $\mathcal{G}$ is
            \begin{equation*}
              \mathcal{N}_r(F)
              = \set{m' \in M \suchthat \Exists e \in E \SuchThat F \rightsemiaction e \ni m'}
              = F \rightsemiaction E
            \end{equation*}
            and the left neighbourhood of $F$ in $\mathcal{G}$ is
            \begin{equation*}
              \mathcal{N}_l(F)
              = \set{m \in M \suchthat \Exists e \in E \SuchThat m \rightsemiaction e \in F}
              = \bigcup_{e \in E} (\blank \rightsemiaction e)^{-1}(F).
            \end{equation*}
            By precondition $\abs{\mathcal{N}_r(F)} = \abs{F \rightsemiaction E} \geq 2 \abs{F}$. Moreover, because $G_0 \in E$,
            we have $F = (\blank \rightsemiaction G_0)^{-1}(F) \subseteq \mathcal{N}_l(F)$ and hence $\abs{\mathcal{N}_l(F)} \geq \abs{F} \geq 2^{-1} \abs{F}$. Therefore, according to
            the Hall harem theorem, there is a perfect $(1,2)$-matching for $\mathcal{G}$. In conclusion, there is a $2$-to-$1$ surjective map $\phi \from M \to M$ such that, for each $m \in M$, we have $(\phi(m), m) \in E$, that is, there is an $e \in E$ such that $\phi(m) \rightsemiaction e = m$.
      \item[\ref{it:tarski-folner:two-to-one-surjective-map} implies \ref{it:tarski-folner:no-paradoxical-decomposition} (see Fig.~\ref{fig:two-to-one-surjective-map-implies-tarski-folner:no-paradoxical-decomposition}).] 
            \begin{figure}[bt]
              \centering
              \begin{tikzpicture}
                \pgfmathsetmacro\w{3.5} 
                \pgfmathsetmacro\cox{7} 

                \foreach \x in {0} {
                  \draw (\x, 0) rectangle (\x + \w, 8);
                  \foreach \y in {1, 2, ..., 7} {
                    \draw (\x, \y) -- (\x + \w, \y);
                  }
                }
                \foreach \x in {\cox} {
                  \draw (\x, 0) rectangle (\x + \w, 8);
                  \foreach \y in {2, 4, 6} {
                    \draw (\x, \y) -- (\x + \w, \y);
                  }
                }
                \foreach \x in {\cox + 0.125} {
                  \draw[gray, dashed] (\x, -0.125) rectangle (\x + \w, 8 - 0.125);
                  \foreach \y in {1.875, 3.875, 5.875} {
                    \draw[gray, dashed] (\x, \y) -- (\x + \w, \y);
                  }
                }
                \draw[very thick] (0, 4) -- (\w, 4);

                \node (M) at (\w / 2, 8.5) {$M$};
                \node (M) at (\cox + \w / 2, 8.5) {$M$};
                \node (psi M) at (-1, 6) {$\psi(M)$};
                \node (psi' M) at (-1, 2) {$\psi'(M)$};

                \node[fill, circle, inner sep = 1pt] (psi m) at (\w - 0.5, 6.5) {};
                \node (psi A n) at (1.45, 6.5) {$\psi(A_e) = A_e \rightsemiaction e$};
                \node[fill, circle, inner sep = 1pt] (psi' m) at (\w - 0.5, 2.5) {};
                \node (psi' B n') at (1.65, 2.5) {$\psi'(B_{e'}) = B_{e'} \rightsemiaction e'$};
                \node[fill, circle, inner sep = 1pt, label = right: $m$] (m) at (\cox + 0.5, 5) {}; 
                \node (A n) at (\cox + \w / 2, 5) {$A_e$};
                \node[gray] (B n') at (\cox + \w /2 + 1.25, 4.5) {$B_{e'}$};

                \draw[->] (psi m) edge[out = -60, in = -180] node[above] {$\phi$} (m)
                          (m) edge[out = 130, in = 10] node[above] {$\psi$} (psi m)
                          (m) edge[out = 70, in = 70] node[above] {$\blank \rightsemiaction e$} (psi m)
                          (psi' m) edge[out = 60, in = -180] node[below] {$\phi$} (m)
                          (m) edge[out = -130, in = -10] node[below] {$\psi'$} (psi' m)
                          (m) edge[out = -70, in = -70] node[below] {$\blank \rightsemiaction e'$} (psi' m);
              \end{tikzpicture}
              \caption{
                Schematic representation of the set-up of the proof of Theorem~\ref{thm:tarski-folner}, Item~\ref{it:tarski-folner:two-to-one-surjective-map} implies Item~\ref{it:tarski-folner:no-paradoxical-decomposition}: Each region enclosed by one of the three columns, two with solid and one with dashed border, is $M$; the dot in the right column, called $m$, is an element of $M$, and the two dots in the left column are its preimages under $\phi$, which are its images under $\psi$ and $\psi'$; there are elements $e$ and $e'$ of $E$ such that $m \rightsemiaction e = \psi(m)$ and $m \rightsemiaction e' = \psi'(m)$, in other words, $m \in A_e$ and $m \in B_{e'}$; as depicted in the right columns with solid and dashed borders, the families $\family{A_e}_{e \in E}$ and $\family{B_{e'}}_{e' \in E}$ are partitions of $M$; as depicted in the left column, the set $\set{\psi(M), \psi'(M)}$ is a partition of $M$, the family $\family{\psi(A_e)}_{e \in E} = \family{A_e \rightsemiaction e}_{e \in E}$ is a partition of $\psi(M)$, and the family $\family{\psi'(B_{e'})}_{e' \in E} = \family{B_{e'} \rightsemiaction e'}_{e' \in E}$ is a partition of $\psi'(M)$.
              }
              \label{fig:two-to-one-surjective-map-implies-tarski-folner:no-paradoxical-decomposition}
            \end{figure}
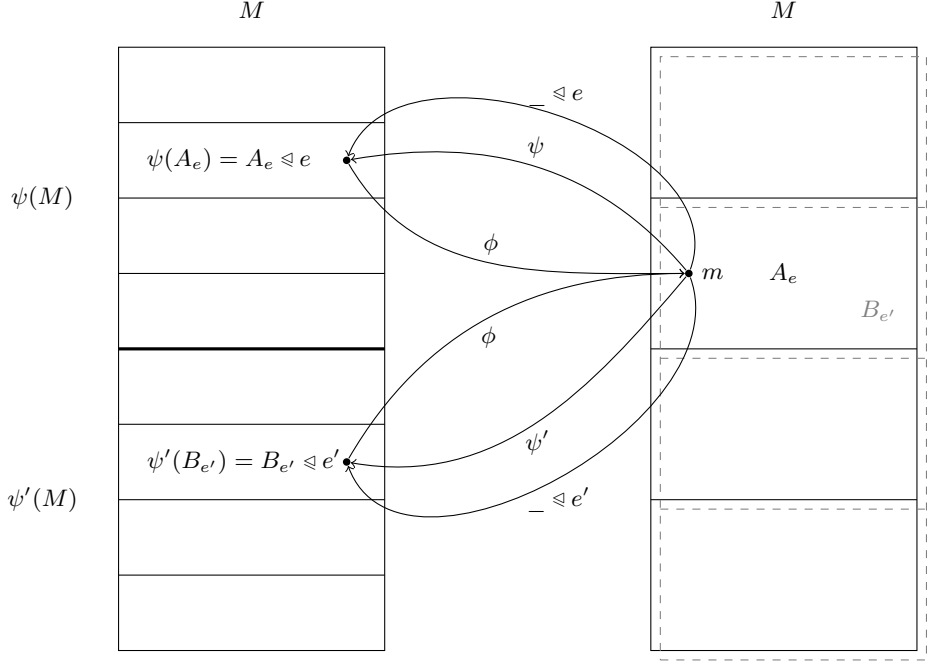
            Let there be a $2$-to-$1$ surjective map $\phi \from M \to M$ and a finite subset $E$ of $G \quotient G_0$ such that
            \begin{equation*}
              \ForEach m \in M \Exists e \in E \SuchThat \phi(m) \rightsemiaction e = m.
            \end{equation*}
            By the axiom of choice, there are two injective maps $\psi$ and $\psi' \from M \to M$ such that, for each $m \in M$, we have $\phi^{-1}(m) = \set{\psi(m), \psi'(m)}$. For each $e \in E$, let
            \begin{equation*}
              A_e = \set{m \in M \suchthat m \rightsemiaction e = \psi(m)}
            \quad\text{and}\quad
              B_e = \set{m \in M \suchthat m \rightsemiaction e = \psi'(m)}.
            \end{equation*}
            Let $m \in M$. There is an $e \in E$ such that $\phi(\psi(m)) \rightsemiaction e = \psi(m)$. Because $\phi(\psi(m)) = m$, we have $m \in A_e$. And, because $\rightsemiaction$ is free, for each $e' \in E \smallsetminus \set{e}$, we have $m \rightsemiaction e' \neq m \rightsemiaction e = \psi(m)$ and thus $m \notin A_{e'}$. Therefore,
            \begin{equation*}
              M = \bigcupdot_{e \in E} A_e
            \quad\text{and analogously}\quad
              M = \bigcupdot_{e \in E} B_e.
            \end{equation*}
            Moreover, $\psi(A_e) = A_e \rightsemiaction e$ and $\psi'(B_e) = B_e \rightsemiaction e$. Hence, because $M = \psi(M) \cupdot \psi'(M)$, and $\psi$ and $\psi'$ are injective,
            \begin{equation*}
              M 
                = \parens*{\bigcupdot_{e \in E} \psi(A_e)} \cupdot \parens*{\bigcupdot_{e \in E} \psi'(B_e)}
                = \parens*{\bigcupdot_{e \in E} A_e \rightsemiaction e} \cupdot \parens*{\bigcupdot_{e \in E} B_e \rightsemiaction e}.
            \end{equation*}
            Furthermore, because $\psi$ and $\psi'$ are injective, for each $e \in E$, the maps $(\blank \rightsemiaction e)\restriction_{A_e} = \psi\restriction_{A_e}$ and $(\blank \rightsemiaction e)\restriction_{B_e} = \psi'\restriction_{B_e}$ are injective. In conclusion, $\ntuple{N, \family{A_e}_{e \in E}, \family{B_e}_{e \in E}}$ is a right paradoxical decomposition of $\mathcal{R}$.
      \item[\ref{it:tarski-folner:no-paradoxical-decomposition} implies \ref{it:tarski-folner:not-right-amenable}.]
            Let there be a right paradoxical decomposition $\ntuple{N, \family{A_e}_{e \in E}, \family{B_e}_{e \in E}}$ of $\mathcal{R}$. According to Lemma~\ref{lem:functional-right-paradoxical-decomposition},
            \begin{equation*}
              \unityfnc_M = \sum_{e \in E} \unityfnc_{A_e}
                          = \sum_{e \in E} \unityfnc_{B_e}
                          = \sum_{e \in E} (\unityfnc_{A_e} \funcamact e) + \sum_{e \in E} (\unityfnc_{B_e} \funcamact e).
            \end{equation*}
            Suppose that $\mathcal{R}$ is right amenable. Then, according to Theorem~\ref{thm:mean-characterisation-of-right-amenable}, there is a $\meanamact$-invariant mean $\nu$ on $M$. Because $\nu$ is linear and normalised,
            \begin{align*}
              1 &= \nu(\unityfnc_M)\\
                &= \sum_{e \in E} \nu(\unityfnc_{A_e} \funcamact e) + \sum_{e \in E} \nu(\unityfnc_{B_e} \funcamact e)\\
                &= \sum_{e \in E} (\nu \meanamact e)(\unityfnc_{A_e}) + \sum_{e \in E} (\nu \meanamact e)(\unityfnc_{B_e})\\
                &= \sum_{e \in E} \nu(\unityfnc_{A_e}) + \sum_{e \in E} \nu(\unityfnc_{B_e})\\
                &= \nu(\unityfnc_M) + \nu(\unityfnc_M)\\
                &= 1 + 1\\
                &= 2,
            \end{align*}
            which contradicts that $1 \neq 2$. In conclusion, $\mathcal{R}$ is not right amenable. \qed 
    \end{description}
  \end{proof}

  \begin{corollary}[Tarski alternative theorem; Alfred Tarski, 1938] 
  \label{cor:tarski}
    Let $\mathcal{M}$ be a left homogeneous space with finite stabilisers. It is right amenable if and only if there is a coordinate system $\mathcal{K}$ for $\mathcal{M}$ such that there is no right paradoxical decomposition of $\ntuple{\mathcal{M}, \mathcal{K}}$. \qed
  \end{corollary}

  \begin{corollary}[Theorem of Følner; Erling Følner, 1955]
  \label{cor:folner}
    Let $\mathcal{M}$ be a left homogeneous space with finite stabilisers. It is right amenable if and only if there is a coordinate system $\mathcal{K}$ for $\mathcal{M}$ such that there is a right Følner net in $\ntuple{\mathcal{M}, \mathcal{K}}$. \qed
  \end{corollary}

  \begin{remark}
    In the situation of Remark~\ref{rem:groups:measureamact}, Corollaries~\ref{cor:tarski} and~\ref{cor:folner} constitute \cite[Theorem~4.9.1]{ceccherini-silberstein:coornaert:2010}.
  \end{remark}


  \section{From Left to Right Amenability}
  \label{sec:left-vs-right}

  \begin{lemma}
  \label{lem:general-sufficient-condition-for-left-implies-right-amenability}
    Let $\mathcal{R} = \ntuple{\ntuple{M, G, \leftaction}, \ntuple{m_0, \family{g_{m_0, m}}_{m \in M}}}$ be a cell space and let $H$ be a subgroup of $G$ such that, for each element $\mathfrak{g} \in G \quotient G_0$, there is an element $h \in H$ such that the maps $\blank \rightsemiaction \mathfrak{g}$ and $h \leftaction \blank$ are inverse to each other.
    If $\ntuple{M, H, \leftaction\restriction_{H \times M}}$ is left amenable, then $\mathcal{R}$ is right amenable.
  \end{lemma}

  \begin{proof}
    Let $\mu \in \pmeasures(M)$. Furthermore, let $\mathfrak{g} \in G \quotient G_0$. There is an $h \in H$ such that $\blank \rightsemiaction \mathfrak{g}$ and $h \leftaction \blank$ are inverse to each other. 
    Moreover, let $A \subseteq M$. Because $\blank \rightsemiaction \mathfrak{g} = (h \leftaction \blank)^{-1} = h^{-1} \leftaction \blank$, we have $A \rightsemiaction \mathfrak{g} = h^{-1} \leftaction A$. Therefore,
    \begin{align*}
      (\mu \measureamact \mathfrak{g})(A)
      &= \mu(A \rightsemiaction \mathfrak{g})\\
      &= \mu(h^{-1} \leftaction A)\\
      &= (h \measureamactleft \mu)(A).
    \end{align*}
    Thus, $\mu \measureamact \mathfrak{g} = h \measureamactleft \mu$. Hence, if $\mu$ is $\measureamactleft\restriction_{H \times [0,1]^{\powerset(M)}}$-invariant, then $\mu$ is $\measureamact$-semi-invariant. In conclusion, if $\ntuple{M, H, \leftaction\restriction_{H \times M}}$ is left amenable, then $\mathcal{R}$ is right amenable. \qed 
  \end{proof}

  \begin{lemma}
  \label{lem:less-general-sufficient-condition-for-left-implies-right-amenability}
    Let $\mathcal{R} = \ntuple{\ntuple{M, G, \leftaction}, \ntuple{m_0, \family{g_{m_0, m}}_{m \in M}}}$ be a cell space and let $H$ be a subgroup of $G$ such that $G = G_0 H$, for each element $\mathfrak{g} \in G \quotient G_0$, the map $\blank \rightsemiaction \mathfrak{g}$ is injective, 
    \begin{equation*}
      \ForEach h \in H \Holds \blank \rightsemiaction h G_0 = h \leftaction \blank,
    \end{equation*}
    and
    \begin{equation*}
      \ForEach h \in H \ForEach \mathfrak{g} \in G \quotient G_0 \Holds (\blank \rightsemiaction h G_0) \rightsemiaction \mathfrak{g} = \blank \rightsemiaction h \cdot \mathfrak{g}.
    \end{equation*}
    If $\ntuple{M, H, \leftaction\restriction_{H \times M}}$ is left amenable, then $\mathcal{R}$ is right amenable.
  \end{lemma}

  \begin{proof}
    Let $g G_0 \in G \quotient G_0$. Because $g^{-1} \in G = G_0 H$, there is a $g_0 \in G_0$ and there is an $h \in H$ such that $g^{-1} = g_0 h$. Thus, $h = g_0^{-1} g^{-1} \in H$. Hence, for each $m \in M$,
    \begin{align*}
      ((\blank \rightsemiaction g G_0) \circ (h \leftaction \blank))(m)
      &= (h \leftaction m) \rightsemiaction g G_0\\
      &= (m \rightsemiaction h G_0) \rightsemiaction g G_0\\
      &= m \rightsemiaction h g G_0\\
      &= m \rightsemiaction g_0^{-1} g^{-1} g G_0\\
      &= m \rightsemiaction G_0\\
      &= m.
    \end{align*}
    Therefore, $h \leftaction \blank$ is right inverse to $\blank \rightsemiaction g G_0$. Hence, $\blank \rightsemiaction g G_0$ is surjective and thus, because it is injective by precondition, bijective. Therefore, $\blank \rightsemiaction g G_0$ and $h \leftaction \blank$ are inverse to each other. In conclusion, according to Lemma~\ref{lem:general-sufficient-condition-for-left-implies-right-amenability}, if $\ntuple{M, H, \leftaction\restriction_{H \times M}}$ is left amenable, then $\mathcal{R}$ is right amenable. \qed
  \end{proof}

  \begin{definition}
    Let $G$ be a group. The set
    \begin{equation*}
      Z(G) = \set{z \in G \suchthat \ForEach g \in G \Holds z g = g z} \mathnote{center $Z(G)$ of $G$}
    \end{equation*}
    is called \define{centre of $G$}.
  \end{definition}

  \begin{lemma}
    Let $G$ be a group. The centre of $G$ is a subgroup of $G$. \qed
  \end{lemma}


  \begin{lemma}
  \label{lem:sufficient-conditions-such-that-left-amenable-implies-right-amenable} 
    Let $\mathcal{R} = \ntuple{\ntuple{M, G, \leftaction}, \ntuple{m_0, \family{g_{m_0, m}}_{m \in M}}}$ be a cell space and let $H$ be a subgroup of $G$ such that $G$ is equal to $G_0 H$, $\leftaction\restriction_{H \times M}$ is free, and $\family{g_{m_0, m}}_{m \in M}$ is included in $Z(H)$. If $\ntuple{M, H, \leftaction\restriction_{H \times M}}$ is left amenable, then $\mathcal{R}$ is right amenable.
  \end{lemma} 

  \begin{proof}
    Let $g \in G$. For each $m \in M$,
    \begin{equation*}
      m \rightsemiaction g G_0
      = g_{m_0, m} g \leftaction m_0
      = g_{m_0, m} \leftaction (g \leftaction m_0).
    \end{equation*}
    Let $m \in M$. For each $m' \in M$, because $\leftaction\restriction_{Z(H) \times M}$ is free and $g_{m_0, m}$, $g_{m_0, m'} \in Z(H)$,
    \begin{align*}
      m' \rightsemiaction g G_0 = m \rightsemiaction g G_0
      &\iff g_{m_0, m'} = g_{m_0, m}\\
      &\iff m' = m.
    \end{align*}
    Therefore, $\blank \rightsemiaction g G_0$ is injective. 

    Let $m \in M$ and let $h \in H$. Because $g_{m_0, m} \in Z(G)$,
    \begin{align*}
      m \rightsemiaction h G_0
      &= g_{m_0, m} h \leftaction m_0\\
      &= h g_{m_0, m} \leftaction m_0\\
      &= h \leftaction m.
    \end{align*}
    Put $m' = m \rightsemiaction h G_0$. Then,
    \begin{align*}
      g_{m_0, m} h \leftaction m_0
      &= h g_{m_0, m} \leftaction m_0\\
      &= h \leftaction m\\
      &= m'.
    \end{align*}
    Hence, because $g_{m_0, m'} \leftaction m_0 = m'$ also and $\leftaction\restriction_{H \times M}$ is free, $g_{m_0, m'} = g_{m_0, m} h$. Therefore,
    \begin{align*}
      (m \rightsemiaction h G_0) \rightsemiaction g G_0
      &= m' \rightsemiaction g G_0\\
      &= g_{m_0, m'} g \leftaction m_0\\
      &= g_{m_0, m} h g \leftaction m_0\\
      &= m \rightsemiaction h g G_0.
    \end{align*}
    In conclusion, according to Lemma~\ref{lem:less-general-sufficient-condition-for-left-implies-right-amenability}, if $\ntuple{M, H, \leftaction\restriction_{H \times M}}$ is left amenable, then $\mathcal{R}$ is right amenable. \qed
  \end{proof}

  \begin{example} 
  \label{ex:one-dimensional-affines-right-amenable}
    Let $M = \K$ be a field, let
    \begin{equation*}
      G = \set{f \from M \to M, x \mapsto a x + b \suchthat a, b \in M, a \neq 0} 
    \end{equation*}
    be the group of affine functions with composition as group multiplication, and let
    \begin{equation*}
      H = \set{f \from M \to M, x \mapsto x + b \suchthat b \in M}
    \end{equation*}
    be the group of translations also with composition as group multiplication. The group $H$ is an abelian subgroup of $G$, which in turn is a non-abelian subgroup of the symmetry group of $M$. Moreover, according to \cite[Example~4.6.2 and Theorem~4.6.3]{ceccherini-silberstein:coornaert:2010}, the group $G$ is left amenable and hence, according to \cite[Proposition~4.5.1]{ceccherini-silberstein:coornaert:2010}, so is its subgroup $H$. Furthermore, the group $G$ acts transitively on $M$ by function application by $\leftaction$ and so does $H$ by $\leftaction\restriction_{H \times M}$, even freely so. Because the groups $G$ and $H$ are left amenable,
    so are the left group sets $\ntuple{M, G, \leftaction}$ and $\ntuple{M, H, \leftaction\restriction_{H \times M}}$. The stabiliser of $m_0 = 0$ is the group of dilations
    \begin{equation*}
      G_0 = \set{f \from M \to M, x \mapsto a x \suchthat a \in M \smallsetminus \set{0}}.
    \end{equation*}
    We have $G = G_0 H$. For each $m \in M$, let
    \begin{align*}
      g_{m_0, m} \from M &\to     M,\\
                       x &\mapsto x + m,
    \end{align*} 
    be the translation by $m$. Then, $\family{g_{m_0, m}}_{m \in M}$ is included in $Z(H) = H$. Hence, according to Lemma~\ref{lem:sufficient-conditions-such-that-left-amenable-implies-right-amenable}, the cell space $\mathcal{R} = \ntuple{\ntuple{M, G, \leftaction}, \ntuple{m_0, \family{g_{m_0, m}}_{m \in M}}}$ is right amenable.
  \end{example}

  \begin{lemma} 
    Let $H$ and $N$ be two groups, let $\phi \from H \to \automorphisms(N)$ be a group homomorphism, let $G$ be the Cartesian product $H \times N$, and let
    \begin{align*} 
      \cdot \from G \times G &\to     G,\\
             ((h,n),(h',n')) &\mapsto (h h', n \phi(h)(n')).
    \end{align*}
    The tuple $(G, \cdot)$ is a group, called \define{semi-direct product of $H$ and $N$ with respect to $\phi$}\graffito{semi-direct product $H \ltimes_\phi N$ of $H$ and $N$ with respect to $\phi$}, and denoted by $H \ltimes_\phi N$. \qed
  \end{lemma} 


  \begin{lemma}
    Let $G$ be a semi-direct product of $H$ and $N$ with respect to $\phi$. The neutral element of $G$ is $(e_H, e_N)$ and, for each element $(h, n) \in G$, the inverse of $(h, n)$ is $(h^{-1}, \phi(h^{-1})(n))$. \qed
  \end{lemma}


  \begin{definition}
    Let $\ntuple{M, G, \leftaction}$ be a left homogeneous space. It is called \define{principal}\graffito{principal} if and only if the action $\leftaction$ is free.
  \end{definition}

  \begin{lemma}
  \label{lem:semi-direct-product-of-group-set-with-stabiliser}
    Let $\ntuple{M, H, \leftaction_H}$ be a principal left homogeneous space. Furthermore, let $G_0$ be a group, let $\phi \from G_0 \to \automorphisms(H)$ be a group homomorphism, let $m_0$ be an element of $M$, for each element $m \in M$, let $h_{m_0,m}$ be the unique element of $H$ such that $h_{m_0,m} \leftaction m_0 = m$, and let
    \begin{align*}
      \leftaction_{G_0} \from G_0 \times M &\to     M,\\
                                  (g_0, m) &\mapsto \phi(g_0)(h_{m_0,m}) \leftaction_H m_0.
    \end{align*}
    Moreover, let $G$ be the semi-direct product of $G_0$ and $H$ with respect to $\phi$, and let
    \begin{align*}
      \leftaction \from G \times M &\to     M,\\
                     ((g_0, h), m) &\mapsto h \leftaction_H (g_0 \leftaction_{G_0} m).
    \end{align*}
    The triple $\ntuple{M, G_0, \leftaction_{G_0}}$ is a left group set and the group $G_0$ is the stabiliser of $m_0$ under $\leftaction_{G_0}$. Furthermore, the tuple $\mathcal{R} = \ntuple{\ntuple{M, G, \leftaction}, \ntuple{m_0, \family{(e_{G_0}, h_{m_0, m})}_{m \in M}}}$ is a cell space and the group $G_0 \times \set{e_H}$ is the stabiliser of $m_0$ under $\leftaction$. Moreover, under the identification of $G_0$ with $G_0 \times \set{e_H}$ and of $H$ with $\set{e_{G_0}} \times H$, the left group sets $\ntuple{M, G_0, \leftaction_{G_0}}$ and $\ntuple{M, H, \leftaction_H}$ are left group subsets of $\ntuple{M, G, \leftaction}$. 
  \end{lemma}

  \begin{proof}
    Because $\phi(e_{G_0}) = \identity_{\automorphisms(H)}$, for each $m \in M$,
    \begin{align*}
      e_{G_0} \leftaction_{G_0} m &= \phi(e_{G_0})(h_{m_0,m}) \leftaction_H m_0\\
                                  &= h_{m_0,m} \leftaction_H m_0\\
                                  &= m.
    \end{align*}
    Let $g_0$ and $g_0' \in G_0$, and let $m \in M$. Because $\leftaction_H$ is free and $h_{m_0, \phi(g_0')(h_{m_0,m}) \leftaction_H m_0} \leftaction_H m_0 = \phi(g_0')(h_{m_0,m}) \leftaction_H m_0$, we have $h_{m_0, \phi(g_0')(h_{m_0,m}) \leftaction_H m_0} = \phi(g_0')(h_{m_0,m})$. Therefore,
    \begin{align*}
      g_0 g_0' \leftaction_{G_0} m &= \phi(g_0 g_0')(h_{m_0,m}) \leftaction_H m_0\\
                                   &= (\phi(g_0) \circ \phi(g_0'))(h_{m_0,m}) \leftaction_H m_0\\
                                   &= \phi(g_0)(\phi(g_0')(h_{m_0,m})) \leftaction_H m_0\\
                                   &= \phi(g_0)(h_{m_0, \phi(g_0')(h_{m_0,m}) \leftaction_H m_0}) \leftaction_H m_0\\
                                   &= g_0 \leftaction_{G_0} (\phi(g_0')(h_{m_0,m}) \leftaction_H m_0)\\
                                   &= g_0 \leftaction_{G_0} (g_0' \leftaction_{G_0} m).
    \end{align*}
    In conclusion, $\ntuple{M, G_0, \leftaction_{G_0}}$ is a left group set.

    Because $h_{m_0, m_0} = e_H$, for each $g_0 \in G_0$,
    \begin{align*}
      g_0 \leftaction_{G_0} m_0 &= \phi(g_0)(e_H) \leftaction_H m_0\\
                                &= e_H \leftaction_H m_0\\
                                &= m_0.
    \end{align*}
    In conclusion, $G_0$ is the stabiliser of $m_0$ under $\leftaction_{G_0}$.

    For each $m \in M$,
    \begin{equation*}
      (e_{G_0}, e_H) \leftaction m = e_H \leftaction_H (e_{G_0} \leftaction_{G_0} m)
                                   = m.
    \end{equation*}
    Let $g_0 \in G_0$, let $h \in H$, and let $m \in M$. Because $h h_{m_0,m} \leftaction_H m_0 = h \leftaction_H m$, we have $h h_{m_0,m} = h_{m_0, h \leftaction_H m}$. Hence,
    \begin{align*}
      \phi(g_0)(h) \leftaction_H (g_0 \leftaction_{G_0} m) &= \phi(g_0)(h) \leftaction_H (\phi(g_0)(h_{m_0,m}) \leftaction_H m_0)\\
                                                           &= \phi(g_0)(h) \phi(g_0)(h_{m_0,m}) \leftaction_H m_0\\
                                                           &= \phi(g_0)(h h_{m_0,m}) \leftaction_H m_0\\
                                                           &= \phi(g_0)(h_{m_0, h \leftaction_H m}) \leftaction_H m_0\\
                                                           &= g_0 \leftaction_{G_0} (h \leftaction_H m).
    \end{align*}
    Therefore, for each $g_0 \in G_0$, each $g_0' \in G_0$, each $h \in H$, each $h' \in H$, and each $m \in M$,
    \begin{align*}
      (g_0, h) (g_0', h') \leftaction m &= (g_0 g_0', h \phi(g_0)(h')) \leftaction m\\
                                        &= h \phi(g_0)(h') \leftaction_H (g_0 g_0' \leftaction_{G_0} m)\\
                                        &= h \leftaction_H \parens[\Big]{\phi(g_0)(h') \leftaction_H \parens[\big]{g_0 \leftaction_{G_0} (g_0' \leftaction_{G_0} m)}}\\
                                        &= h \leftaction_H \parens[\Big]{g_0 \leftaction_{G_0} \parens[\big]{h' \leftaction_H (g_0' \leftaction_{G_0} m)}}\\
                                        &= (g_0, h) \leftaction \parens[\big]{h' \leftaction_H (g_0' \leftaction_{G_0} m)}\\
                                        &= (g_0, h) \leftaction \parens[\big]{(g_0', h') \leftaction m}.
    \end{align*}
    In conclusion, $\ntuple{M, G, \leftaction}$ is a left group action.

    Because $\leftaction_H$ is transitive and, for each $h \in H$ and each $m \in M$, we have $(e_{G_0}, h) \leftaction m = h \leftaction m$, the left group action $\leftaction$ is transitive and hence $\mathcal{M} = \ntuple{M, G, \leftaction}$ is a left homogeneous space. Moreover, because, for each $m \in M$,
    \begin{align*}
      (e_{G_0}, h_{m_0, m}) \leftaction m_0
      &= h_{m_0, m} \leftaction_H (e_{G_0} \leftaction_{G_0} m_0)\\
      &= h_{m_0, m} \leftaction_H m_0\\
      &= m,
    \end{align*}
    the tuple $\mathcal{K} = \ntuple{m_0, \family{(e_{G_0}, h_{m_0, m})}_{m \in M}}$ is a coordinate system for $\mathcal{M}$. Therefore, $\mathcal{R} = \ntuple{\mathcal{M}, \mathcal{K}}$ is a cell space.

    Because $G_0$ is the stabiliser of $m_0$ under $\leftaction_{G_0}$, for each $(g_0, h) \in G$, we have $(g_0, h) \leftaction m_0 = h \leftaction_H (g_0 \leftaction m_0) = h \leftaction m_0$. Because $\leftaction_H$ is free, $G_0 \times \set{e_H}$ is the stabiliser of $m_0$ under $\leftaction$.

    Under the identification of $G_0$ with $G_0 \times \set{e_H}$ and of $H$ with $\set{e_{G_0}} \times H$, we have $\leftaction\restriction_{G_0 \times M} = \leftaction_{G_0}$ and $\leftaction\restriction_{H \times M} = \leftaction_H$. \qed 
  \end{proof}

  \begin{corollary}
  \label{cor:sufficient-conditions-such-that-left-amenable-implies-right-amenable}
    In the situation of Lemma~\ref{lem:semi-direct-product-of-group-set-with-stabiliser}, let $H$ be abelian. The cell space $\mathcal{R}$ is right amenable.
  \end{corollary}

  \begin{proof} 
    According to \cite[Theorem~4.6.1]{ceccherini-silberstein:coornaert:2010}, because $H$ is abelian, it is left amenable. Therefore, $\ntuple{M, H, \leftaction_H}$ is left amenable. Identify $G_0$ with $G_0 \times \set{e_H}$ and identify $H$ with $\set{e_{G_0}} \times H$. Then, $H$ is a subgroup of $G$, and $G = G_0 H$, and $\leftaction\restriction_{H \times M} = \leftaction_H$ is free, and, for each $m \in M$, we have $(e_{G_0}, h_{m_0, m}) \in H = Z(H)$. Hence, according to Lemma~\ref{lem:sufficient-conditions-such-that-left-amenable-implies-right-amenable}, the cell space $\mathcal{R}$ is right amenable. \qed
  \end{proof}

  \begin{example} 
    Let $d$ be a positive integer; let $E$ be the $d$-dimensional Euclidean group, that is, the symmetry group of the $d$-dimensional Euclidean space, in other words, the isometries of $\R^d$ with respect to the Euclidean metric with function composition; let $T$ be the $d$-dimensional translation group; and let $O$ be the $d$-dimensional orthogonal group. The group $T$ is abelian, a normal subgroup of $E$, and isomorphic to $\R^d$ with addition; the group $O$ is isomorphic to the quotient $E \quotient T$ and to the ($d \times d$)-dimensional orthogonal matrices with matrix multiplication; the group $E$ is isomorphic to the semi-direct product $O \ltimes_\iota T$, where $\iota \from O \to \automorphisms(\R^d)$ is the inclusion map. The groups $T$, $O$, and $E$ act on $\R^d$ on the left by function application, denoted by $\leftaction_T$, $\leftaction_O$, and $\leftaction$, respectively; under the identification of $T$ with $\R^d$ by $t \mapsto [v \mapsto v + t]$, of $O$ with the orthogonal matrices of $\R^{d \times d}$ by $A \mapsto [v \mapsto A v]$, and of $E$ with $O \ltimes_\iota T$ by $(A, t) \mapsto [v \mapsto A v + t]$, we have 
    \begin{align*}
      \leftaction_T \from T \times \R^d &\to     \R^d,\\
                                 (t, v) &\mapsto v + t,
    \end{align*}
    and
    \begin{align*}
      \leftaction_O \from O \times \R^d &\to     \R^d,\\
                                 (A, v) &\mapsto A v,
    \end{align*}
    and
    \begin{align*}
      \leftaction \from E \times \R^d &\to     \R^d,\\
                          ((A, t), v) &\mapsto A v + t,
    \end{align*}
    and
    \begin{align*}
      \iota \from O &\to     \automorphisms(\R^d),\\
                  A &\mapsto [v \mapsto A v].
    \end{align*}
    Hence, for each vector $v \in \R^d$, we have $v \leftaction_T 0 = v$, therefore, $\leftaction_O = [(A, v) \mapsto \iota(A)(v) \leftaction_T 0]$, and thus $\leftaction = [((A, t), v) \mapsto t \leftaction_T (A \leftaction_O v)]$. Moreover, because the group $(T, \after) \isomorphic (\R^d, +)$ is abelian, according to \cite[Theorem~4.6.1]{ceccherini-silberstein:coornaert:2010}, it is left amenable and so is $\ntuple{\R^d, \R^d, +} \isomorphic \ntuple{\R^d, T, \leftaction}$. In conclusion, according to Corollary~\ref{cor:sufficient-conditions-such-that-left-amenable-implies-right-amenable}, the cell space $\ntuple{\ntuple{\R^d, E, \leftaction}, \ntuple{0, \family{-v}_{v \in \R^d}}}$ is right amenable.
  \end{example} 

%

  \newpage

  \appendix

  In Appendix~\ref{apx:topologies} we present the basic theory of topologies. In Appendix~\ref{apx:nets} we present the basic theory of nets. In Appendix~\ref{apx:initial-and-product-topologies} we introduce initial and product topologies. In Appendix~\ref{apx:compactness} we introduce the notion of compactness for topological spaces. In Appendix~\ref{apx:dual-spaces} we introduce topological dual spaces of topological spaces. In Appendix~\ref{apx:hall} we state Hall's marriage and harem theorems. And in Appendix~\ref{apx:zorn} we state Zorn's lemma.

  \section{Topologies}
  \label{apx:topologies}

  The theory of topologies and nets as presented here may be found in more detail in Appendix~A in the monograph \enquote{Cellular Automata and Groups}\cite{ceccherini-silberstein:coornaert:2010}.

  \begin{definition}
    Let $X$ be a set and let $\mathcal{T}$ be a set of subsets of $X$. The set $\mathcal{T}$ is called \define{topology on $X$}\graffito{topology $\mathcal{T}$ on $X$} if, and only if
    \begin{enumerate}
      \item $\set{\emptyset, X}$ is a subset of $\mathcal{T}$,
      \item for each family $\family{O_i}_{i \in I}$ of elements in $\mathcal{T}$, the union $\bigcup_{i \in I} O_i$ is an element of $\mathcal{T}$,
      \item for each finite family $\family{O_i}_{i \in I}$ of elements in $\mathcal{T}$, the intersection $\bigcap_{i \in I} O_i$ is an element of $\mathcal{T}$. 
    \end{enumerate}
  \end{definition}

  \begin{definition}
    Let $X$ be a set, and let $\mathcal{T}$ and $\mathcal{T}'$ be two topologies on $X$. The topology $\mathcal{T}$ is called
    \begin{enumerate}
      \item \define{coarser than $\mathcal{T}'$}\graffito{coarser than $\mathcal{T}'$} if, and only if $\mathcal{T} \subseteq \mathcal{T}'$;
      \item \define{finer than $\mathcal{T}'$}\graffito{finer than $\mathcal{T}'$} if, and only if $\mathcal{T} \supseteq \mathcal{T}'$.
    \end{enumerate}
  \end{definition}

  \begin{definition} 
    Let $X$ be a set and let $\mathcal{T}$ be a topology on $X$. The tuple $(X, \mathcal{T})$ is called \define{topological space}\graffito{topological space $(X, \mathcal{T})$}, each subset $O$ of $X$ with $O \in \mathcal{T}$ is called \define{open in $X$}\graffito{open set $O$ in $X$}, each subset $A$ of $X$ with $X \smallsetminus A \in \mathcal{T}$ is called \define{closed in $X$}\graffito{closed set $A$ in $X$}, and each subset $U$ of $X$ that is both open and closed is called \define{clopen in $X$}\graffito{clopen set $U$ in $X$}. 

    The set $X$ is said to be \define{equipped with $\mathcal{T}$}\graffito{equipped with $\mathcal{T}$} if, and only if it shall be implicitly clear that $\mathcal{T}$ is the topology on $X$ being considered. The set $X$ is called \define{topological space}\graffito{topological space $X$} if, and only if it is implicitly clear what topology on $X$ is being considered. 
  \end{definition}

  \begin{example}
    Let $X$ be a set. The set $\powerset(X)$ is the finest topology on $X$. Itself as well as the topological space $(X, \powerset(X))$ are called \define{discrete}\index{discrete!topological space}\index{discrete!topology}\graffito{discrete}.
  \end{example}

  \begin{definition} 
    Let $(X, \mathcal{T})$ be a topological space, let $x$ be a point of $X$, and let $N$ be a subset of $X$. The set $N$ is called \define{neighbourhood of $x$}\graffito{neighbourhood of $x$} if, and only if there is an open subset $O$ of $X$ such that $x \in O$ and $O \subseteq N$.
  \end{definition}

  \begin{definition} 
    Let $(X, \mathcal{T})$ be a topological space and let $x$ be a point of $X$. The set of all open neighbourhoods of $x$ is denoted by $\mathcal{T}_x$. 
  \end{definition}

  \section{Nets}
  \label{apx:nets}

  \begin{definition} 
    Let $I$ be a set and let $\leq$ be a binary relation on $I$.
    The relation $\leq$ is called \define{preorder on $I$}\graffito{preorder $\leq$ on $I$} and the tuple $(I, \leq)$ is called \define{preordered set}\graffito{preordered set $(I, \leq)$} if, and only if the relation $\leq$ is reflexive and transitive.
  \end{definition}

  \begin{definition}
    Let $\leq$ be a preorder on $I$. It is called \define{directed}\graffito{directed preorder $\leq$ on $I$}\index{preorder on $I$!directed} and the preordered set $(I, \leq)$ is called \define{directed set}\graffito{directed set $(I, \leq)$} if, and only if
    \begin{equation*}
      \ForEach i \in I \ForEach i' \in I \Exists i'' \in I \SuchThat i \leq i'' \land i' \leq i''.
    \end{equation*}
  \end{definition}

  \begin{definition} 
    Let $\leq$ be a preorder on $I$, let $J$ be a subset of $I$, and let $i$ be an element of $I$. The element $i$ is called \define{upper bound of $J$ in $(I, \leq)$}\graffito{upper bound of $J$ in $(I, \leq)$} if, and only if
    \begin{equation*}
      \ForEach i' \in J \Holds i' \leq i.
    \end{equation*}
  \end{definition}

  \begin{definition} 
    Let $M$ be a set, let $I$ be a set, and let $f \from I \to M$ be a map. The map $f$ is called \define{family of elements in $M$ indexed by $I$}\graffito{family $\family{m_i}_{i \in I}$ of elements in $M$ indexed by $I$} and denoted by $\family{m_i}_{i \in I}$, where, for each index $i \in I$, $m_i = f(i)$.
  \end{definition}

  \begin{definition}
    Let $I$ be a set, let $\leq$ be a binary relation on $I$, and let $\family{m_i}_{i \in I}$ be a family of elements in $M$ indexed by $I$. The family $\family{m_i}_{i \in I}$ is called \define{net in $M$ indexed by $(I, \leq)$}\graffito{net $\family{m_i}_{i \in I}$ in $M$ indexed by $(I, \leq)$} if, and only if the tuple $(I, \leq)$ is a directed set.
  \end{definition}

  \begin{definition} 
    Let $\net{m_i}_{i \in I}$ and $\net{m_j'}_{j \in J}$ be two nets in $M$. The net $\net{m_j'}_{j \in J}$ is called \define{subnet of $\net{m_i}_{i \in I}$}\graffito{subnet $\net{m_j'}_{j \in J}$ of $\net{m_i}_{i \in I}$} if, and only if there is a map $f \from J \to I$ such that $\net{m_j'}_{j \in J} = \net{m_{f(j)}}_{j \in J}$ and
    \begin{equation*} 
      \ForEach i \in I \Exists j \in J \SuchThat \ForEach j' \in J \Holds (j' \geq j \implies f(j') \geq i).
    \end{equation*}
  \end{definition}

  \begin{definition}
    Let $(X, \mathcal{T})$ be a topological space, let $\net{x_i}_{i \in I}$ be a net in $X$ indexed by $(I, \leq)$, and let $x$ be a point of $X$. The net $\net{x_i}_{i \in I}$ is said to \define{converge to $x$}\graffito{converge to $x$} and $x$ is called \define{limit point of $\net{x_i}_{i \in I}$}\graffito{limit point of $\net{x_i}_{i \in I}$} if, and only if
    \begin{equation*}
      \ForEach O \in \mathcal{T}_x \Exists i_0 \in I \SuchThat \ForEach i \in I \Holds (i \geq i_0 \implies x_i \in O).
    \end{equation*}
  \end{definition} 

  \begin{definition}
    Let $(X, \mathcal{T})$ be a topological space and let $\net{x_i}_{i \in I}$ be a net in $X$ indexed by $(I, \leq)$. The net $\net{x_i}_{i \in I}$ is called \define{convergent}\graffito{convergent} if, and only if there is a point $x \in X$ such that it converges to $x$.
  \end{definition}

  \begin{remark} 
    Let $\net{m_i}_{i \in I}$ be a net that converges to $x$. Each subnet $\net{m_j'}_{j \in J}$ of $\net{m_i}_{i \in I}$ converges to $x$.
  \end{remark}

  \begin{lemma} 
    Let $(X, \mathcal{T})$ be a topological space, let $Y$ be a subset of $X$, and let $x$ be an element of $X$. Then, $x \in \closure{Y}$ if, and only if there is a net $\net{y_i}_{i \in I}$ in $Y$ that converges to $x$. 
  \end{lemma}

  \begin{proof}
    See Proposition~A.2.1 in \enquote{Cellular Automata and Groups}\cite{ceccherini-silberstein:coornaert:2010}. \qed
  \end{proof}

  \begin{lemma} 
    Let $(X, \mathcal{T})$ be a topological space. It is Hausdorff if, and only if each convergent net in $X$ has exactly one limit point. 
  \end{lemma}

  \begin{proof}
    See Proposition~A.2.2 in \enquote{Cellular Automata and Groups}\cite{ceccherini-silberstein:coornaert:2010}. \qed
  \end{proof}

  \begin{definition}
    Let $(X, \mathcal{T})$ be a Hausdorff topological space, let $\net{x_i}_{i \in I}$ be a convergent net in $X$ indexed by $(I, \leq)$, and let $x$ be the limit point of $\net{x_i}_{i \in I}$. The point $x$ is denoted by $\lim_{i \in I} x_i$\graffito{$\lim_{i \in I} x_i$} and we write $x_i \underset{i \in I}{\to} x$\graffito{$x_i \underset{i \in I}{\to} x$}.
  \end{definition}

  \begin{definition}
    Let $(X, \mathcal{T})$ be a topological space, let $\net{x_i}_{i \in I}$ be a net in $X$ indexed by $(I, \leq)$, and let $x$ be an element of $X$. The point $x$ is called \define{cluster point of $\net{x_i}_{i \in I}$}\graffito{cluster point $x$ of $\net{x_i}_{i \in I}$} if, and only if
    \begin{equation*}
      \ForEach O \in \mathcal{T}_x \ForEach i \in I \Exists i' \in I \SuchThat (i' \geq i \land x_{i'} \in O).
    \end{equation*}
  \end{definition}

  \begin{lemma} 
    Let $(X, \mathcal{T})$ be a topological space, let $\net{x_i}_{i \in I}$ be a net in $X$ indexed by $(I, \leq)$, and let $x$ be an element of $X$. The point $x$ is a cluster point of $\net{x_i}_{i \in I}$ if, and only if there is a subnet of $\net{x_i}_{i \in I}$ that converges to $x$.
  \end{lemma}

  \begin{proof}
    See Proposition~A.2.3 in \enquote{Cellular Automata and Groups}\cite{ceccherini-silberstein:coornaert:2010}. \qed
  \end{proof}

  \begin{lemma}
    Let $(X, \mathcal{T})$ and $(X', \mathcal{T}')$ be two topological spaces, let $f$ be a continuous map from $X$ to $X'$, let $\net{x_i}_{i \in I}$ be a net in $X$, and let $x$ be an element of $X$.
    \begin{enumerate}
      \item If $x$ is a limit point of $\net{x_i}_{i \in I}$, then $f(x)$ is a limit point of $\net{f(x_i)}_{i \in I}$. 
      \item If $x$ is a cluster point of $\net{x_i}_{i \in I}$, then $f(x)$ is a cluster point of $\net{f(x_i)}_{i \in I}$.
    \end{enumerate}
  \end{lemma}

  \begin{proof} 
    Confer the last paragraph of Sect.~A.2 in \enquote{Cellular Automata and Groups}\cite{ceccherini-silberstein:coornaert:2010}. \qed
  \end{proof}

  \begin{definition} 
    Let $\overline{\R} = \R \cup \set{-\infty,+\infty}$ be the affinely extended real numbers and let $\net{r_i}_{i \in I}$ be a net in $\overline{\R}$ indexed by $(I, \leq)$.
    \begin{enumerate} 
      \item The limit of the net $\net{\inf_{i' \geq i} r_{i'}}_{i \in I}$ is called \define{limit inferior of $\net{r_i}_{i \in I}$}\graffito{limit inferior $\liminf_{i \in I} r_i$ of $\net{r_i}_{i \in I}$} and denoted by $\liminf_{i \in I} r_i$.
      \item The limit of the net $\net{\sup_{i' \geq i} r_{i'}}_{i \in I}$ is called \define{limit superior of $\net{r_i}_{i \in I}$}\graffito{limit superior $\limsup_{i \in I} r_i$ of $\net{r_i}_{i \in I}$} and denoted by $\limsup_{i \in I} r_i$.
    \end{enumerate}
  \end{definition}

  \section{Initial and Product Topologies}
  \label{apx:initial-and-product-topologies}

  \begin{definition} 
    Let $X$ be a set, let $I$ be a set, and, for each index $i \in I$, let $(Y_i, \mathcal{T}_i)$ be a topological space and let $f_i$ be a map from $X$ to $Y_i$. The coarsest topology on $X$ such that, for each index $i \in I$, the map $f_i$ is continuous, is called \define{initial with respect to $\family{f_i}_{i \in I}$}\graffito{initial with respect to $\family{f_i}_{i \in I}$}. 
  \end{definition}

  \begin{lemma} 
  \label{lem:map-to-initial-topology-continuous-iff-gens-after-map-continuous}
    Let $(X, \mathcal{T})$ be a topological space, where $\mathcal{T}$ is the initial topology with respect to $\family{f_i \from X \to Y_i}_{i \in I}$, let $(Z, \mathcal{S})$ be a topological space, and let $g$ be a map from $Z$ to $X$. The map $g$ is continuous if, and only if, for each index $i \in I$, the map $f_i \after g$ is continuous. \qed
  \end{lemma}


  \begin{lemma}
    Let $(X, \mathcal{T})$ be a topological space, where $\mathcal{T}$ is the initial topology with respect to $\family{f_i \from X \to Y_i}_{i \in I}$, let $\net{x_{i'}}_{i' \in I'}$ be a net in $X$, and let $x$ be a point in $X$. The point $x$ is a limit point or cluster point of $\net{x_{i'}}_{i' \in I'}$ if, and only if, for each index $i \in I$, the point $f_i(x)$ is a limit point or cluster point, respectively, of $\net{f_i(x_{i'})}_{i' \in I'}$.
  \end{lemma}

  \begin{proof}
    Confer the last paragraph of Sect.~A.3 in \enquote{Cellular Automata and Groups}\cite{ceccherini-silberstein:coornaert:2010}. \qed
  \end{proof}

  \begin{definition} 
    Let $\family{(X_i, \mathcal{T}_i)}_{i \in I}$ be a family of topological spaces, let $X$ be the set $\prod_{i \in I} X_i$, and, for each index $i \in I$, let $\pi_i$ be the projection of $X$ onto $X_i$. The initial topology on $X$ with respect to $\family{\pi_i}_{i \in I}$ is called \define{product}\graffito{product}.
  \end{definition}

  \begin{remark} 
    The product topology on $X$ has for a base the sets $\prod_{i \in I} O_i$, where, for each index $i \in I$, the set $O_i$ is an open subset of $X_i$, and the set $\set{i \in I \suchthat O_i \neq X_i}$ is finite.
  \end{remark}

  \begin{definition} 
    Let $\family{(X_i, \mathcal{T}_i)}_{i \in I}$ be a family of discrete topological spaces and let $X$ be the set $\prod_{i \in I} X_i$. The product topology on $X$ is called \define{prodiscrete}\graffito{prodiscrete}. 
  \end{definition}

  \begin{lemma} 
    Let $\family{(X_i, \mathcal{T}_i)}_{i \in I}$ be a family of Hausdorff topological spaces. The set $\prod_{i \in I} X_i$, equipped with the product topology, is Hausdorff.
  \end{lemma}

  \begin{proof}
    See Proposition~A.4.1 in \enquote{Cellular Automata and Groups}\cite{ceccherini-silberstein:coornaert:2010}. \qed
  \end{proof}

  \begin{definition}
    Let $X$ be a topological space. It is called \define{totally disconnected}\graffito{totally disconnected} if, and only if, for each non-empty and connected subset $A$ of $X$, we have $\abs{A} = 1$. 
  \end{definition}

  \begin{lemma}
    Let $\family{(X_i, \mathcal{T}_i)}_{i \in I}$ be a family of totally disconnected topological spaces. The set $\prod_{i \in I} X_i$, equipped with the product topology, is totally disconnected.
  \end{lemma}

  \begin{proof}
    See Proposition~A.4.2 in \enquote{Cellular Automata and Groups}\cite{ceccherini-silberstein:coornaert:2010}. \qed
  \end{proof}

  \begin{lemma}
    Let $\family{(X_i, \mathcal{T}_i)}_{i \in I}$ be a family of topological spaces and, for each index $i \in I$, let $A_i$ be a closed subset of $X_i$. The set $\prod_{i \in I} A_i$ is a closed subset of $\prod_{i \in I} X_i$, equipped with the product topology.
  \end{lemma}

  \begin{proof}
    See Proposition~A.4.3 in \enquote{Cellular Automata and Groups}\cite{ceccherini-silberstein:coornaert:2010}. \qed
  \end{proof}

  \section{Compactness}
  \label{apx:compactness}

  \begin{definition} 
    Let $(X, \mathcal{T})$ be a topological space and let $\family{O_i}_{i \in I}$ be a family of elements of $\mathcal{T}$. The family $\family{O_i}_{i \in I}$ is called \define{open cover of $X$}\graffito{open cover $\family{O_i}_{i \in I}$ of $X$} if, and only if, $\bigcup_{i \in I} O_i = X$.
  \end{definition}

  \begin{definition} 
    Let $(X, \mathcal{T})$ be a topological space. It is called \define{compact}\graffito{compact} if, and only if, for each open cover $\family{O_i}_{i \in I}$ of $X$, there is a finite subset $J$ of $I$ such that $\family{O_j}_{j \in J}$ is an open cover of $X$. 
  \end{definition}

  \begin{lemma}
    Let $(X, \mathcal{T})$ be a topological space. It is compact if, and only if, for each family $\family{A_i}_{i \in I}$ of closed subsets of $X$ such that, for each finite subset $J$ of $I$, $\bigcap_{j \in J} A_j \neq \emptyset$, we have $\bigcup_{i \in I} A_i \neq \emptyset$. 
  \end{lemma}

  \begin{proof} 
    Confer first paragraph of Sect.~A.5 in \enquote{Cellular Automata and Groups}\cite{ceccherini-silberstein:coornaert:2010}. \qed
  \end{proof}

  \begin{theorem} 
    Let $(X, \mathcal{T})$ be a topological space. The following statements are equivalent:
    \begin{enumerate}
      \item The space $(X, \mathcal{T})$ is compact; 
      \item Each net in $X$ has a cluster point with respect to $\mathcal{T}$; 
      \item Each net in $X$ has a convergent subnet with respect to $\mathcal{T}$.
    \end{enumerate}
  \end{theorem}

  \begin{proof}
    See Theorem~A.5.1 in \enquote{Cellular Automata and Groups}\cite{ceccherini-silberstein:coornaert:2010}. \qed
  \end{proof}

  \begin{theorem}[Andrey Nikolayevich Tikhonov, 1935] 
  \label{thm:tychonoff}
    Let $\family{(X_i, \mathcal{T}_i)}_{i \in I}$ be a family of compact topological spaces. The set $\prod_{i \in I} X_i$, equipped with the product topology, is compact.
  \end{theorem}

  \begin{proof}
    See Theorem~A.5.2 in \enquote{Cellular Automata and Groups}\cite{ceccherini-silberstein:coornaert:2010}. \qed
  \end{proof}

  \begin{corollary} 
  \label{cor:tychonoff}
    Let $\family{(X_i, \mathcal{T}_i)}_{i \in I}$ be a family of finite topological spaces. The set $\prod_{i \in I} X_i$, equipped with the product topology, is compact.
  \end{corollary}

  \begin{proof}
    See paragraph before Corollary~A.5.3 in \enquote{Cellular Automata and Groups}\cite{ceccherini-silberstein:coornaert:2010}. \qed
  \end{proof}

  \section{Dual Spaces}
  \label{apx:dual-spaces}

  The theory of dual spaces as presented here may be found in more detail in Appendix~F in \enquote{Cellular Automata and Groups}\cite{ceccherini-silberstein:coornaert:2010}.

  In this section, let $(X, \norm{\blank})$ be a normed $\R$-vector space.

  \begin{definition} 
    The vector space 
    \begin{equation*}
      X^* = \set{\psi \from X \to \R \text{ linear} \suchthat \psi \text{ is continuous}} \mathnote{topological dual space $X^*$ of $X$} 
    \end{equation*}
    is called \define{topological dual space of $X$}.
  \end{definition}

  \begin{definition}
    \begin{enumerate}
      \item The norm
            \begin{align*}
              \norm{\blank}_{X^*} \from X^* &\to     \R, \mathnote{operator norm $\norm{\blank}_{X^*}$ on $X^*$}\\
                                       \psi &\mapsto \sup_{x \in X \smallsetminus \set{0}} \frac{\abs{\psi(x)}}{\norm{x}},
            \end{align*}
            is called \define{operator norm on $X^*$}.
      \item The topology on $X^*$ induced by $\norm{\blank}_{X^*}$ is called \define{strong topology on $X^*$}\graffito{strong topology on $X^*$}.
    \end{enumerate}
  \end{definition}

  \begin{definition}
    Let $x$ be an element of $X$. The map
    \begin{align*}
      \ev_x \from X^* &\to     \R, \mathnote{evaluation map $\ev_x$ at $x$}\\
                 \psi &\mapsto \psi(x),
    \end{align*}
    is called \define{evaluation map at $x$}.
  \end{definition}

  \begin{definition} 
    The initial topology on $X^*$ with respect to $\family{\ev_x}_{x \in X}$ is called \define{weak-$*$ topology on $X^*$}\graffito{weak-$*$ topology on $X^*$}.
  \end{definition}

  \begin{lemma}
    Let $\net{\psi_i}_{i \in I}$ be a net in $X^*$, let $\psi$ be an element of $X^*$, and let $X^*$ be equipped with the weak-$*$ topology. The net $\net{\psi_i}_{i \in I}$ converges to $\psi$ if, and only if, for each element $x$ of $X$, the net $\net{\psi_i(x)}_{i \in I}$ converges to $\psi(x)$. \qed
  \end{lemma}


  \begin{lemma}
  \label{lem:weak-star-coarser-than-strong-topology}
    The weak-$*$ topology on $X^*$ is coarser than the strong topology on $X^*$. \qed
  \end{lemma}


  \begin{corollary}
    Let $\net{\psi_i}_{i \in I}$ be a net in $X^*$ that converges to $\psi$ with respect to the strong topology on $X^*$. The net $\net{\psi_i}_{i \in I}$ converges to $\psi$ with respect to the weak-$*$ topology on $X^*$. \qed
  \end{corollary}


  \begin{definition} 
    Let $Y$ be a subset of $X$. The set $Y$ is called
            \define{convex}\graffito{convex} if, and only if,
            \begin{equation*}
              \ForEach (y, y') \in Y \times Y \ForEach t \in [0,1] \Holds t y + (1 - t) y' \in Y. 
            \end{equation*}
  \end{definition}

  \begin{definition} 
    The topological vector space $X$ is called \define{locally convex} if, and only if, the origin has a neighbourhood base of convex sets. 
  \end{definition}

  \begin{lemma}
    Let $X^*$ be equipped with the weak-$*$ topology and let $\psi$ be an element of $X^*$. A neighbourhood base of $\psi$ is given by the sets
      \begin{multline*}
        B(\psi, F, \varepsilon) = \set{\psi' \in X^* \suchthat \ForEach x \in F \Holds \abs{\psi(x) - \psi'(x)} < \varepsilon},\\
        \text{ for } F \subseteq X \text{ finite and } \varepsilon \in \R_{> 0}.
        \mathnote{$B(\psi, F, \varepsilon)$, for $\psi \in X^*$, $F \subseteq X$ finite, $\varepsilon \in \R_{> 0}$} \qed
      \end{multline*}
  \end{lemma}


  \begin{corollary}
    Let $X^*$ be equipped with the weak-$*$ topology. The space $X^*$ is locally convex. \qed
  \end{corollary}


  \begin{lemma}
  \label{lem:weak-star-topology-is-Hausdorff}
    The space $X^*$, equipped with the weak-$*$ topology, is Hausdorff.
  \end{lemma}

  \begin{proof}
    Confer last paragraph of Sect.~F.2 in \enquote{Cellular Automata and Groups}\cite{ceccherini-silberstein:coornaert:2010}. \qed
  \end{proof}

  \begin{theorem}[Stefan Banach, 1932; Leonidas Alaoglu, 1940]
  \label{thm:banach-alaoglu}
    Let $X^*$ be equipped with the weak-$*$ topology. The unit ball $\set{\psi \in X^* \suchthat \norm{\psi}_{X^*} \leq 1}$, equipped with the subspace topology, is compact.
  \end{theorem}

  \begin{proof}
    See Theorem~F.3.1 in \enquote{Cellular Automata and Groups}\cite{ceccherini-silberstein:coornaert:2010}. \qed
  \end{proof}

  \section{Hall's Theorems}
  \label{apx:hall}

  The theory concerning Hall's theorems as presented here may be found in more detail in Appendix~H in the monograph \enquote{Cellular Automata and Groups}\cite{ceccherini-silberstein:coornaert:2010}.

  \begin{definition}
    Let $X$ and $Y$ be two sets, and let $E$ be a subset of $X \times Y$. The triple $(X, Y, E)$ is called \define{bipartite graph}\graffito{bipartite graph $(X, Y, E)$}, each element $x$ of $X$ is called \define{left vertex}\graffito{left vertex $x$}, each element $y$ of $Y$ is called \define{right vertex}\graffito{right vertex $y$}, and each element $e$ of $E$ is called \define{edge}\graffito{edge $e$}.
  \end{definition}

  \begin{definition}
    Let $(X, Y, E)$ and $(X', Y', E')$ be two bipartite graphs. The graph $(X, Y, E)$ is called \define{bipartite subgraph of $(X', Y', E')$}\graffito{bipartite subgraph $(X, Y, E)$ of $(X', Y', E')$} if, and only if, $X \subseteq X'$, $Y \subseteq Y'$, and $E \subseteq E'$.
  \end{definition}

  \begin{definition} 
    Let $(X, Y, E)$ be a bipartite graph, and let $(x, y)$ and $(x', y')$ be two elements of $E$. The edges $(x, y)$ and $(x', y')$ are called \define{adjacent}\graffito{adjacent} if, and only if, $x = x'$ or $y = y'$.
  \end{definition}

  \begin{definition}
    Let $(X, Y, E)$ be a bipartite graph.
    \begin{enumerate}
      \item Let $x$ be an element of $X$. The set
            \begin{equation*}
              \mathcal{N}_r(x) = \set{y \in Y \suchthat (x, y) \in E} \mathnote{right neighbourhood $\mathcal{N}_r(x)$ of $x$}
            \end{equation*}
            is called \define{right neighbourhood of $x$}.
      \item Let $A$ be a subset of $X$. The set $\bigcup_{a \in A} \mathcal{N}_r(a)$ is called \define{right neighbourhood of $A$}\graffito{right neighbourhood $\mathcal{N}_r(A)$ of $A$} and is denoted by $\mathcal{N}_r(A)$. 
      \item Let $y$ be an element of $Y$. The set
            \begin{equation*}
              \mathcal{N}_l(y) = \set{x \in X \suchthat (x, y) \in E} \mathnote{left neighbourhood $\mathcal{N}_l(y)$ of $y$}
            \end{equation*}
            is called \define{left neighbourhood of $y$}.
      \item Let $B$ be a subset of $Y$. The set $\bigcup_{b \in B} \mathcal{N}_l(b)$ is called \define{left neighbourhood of $B$}\graffito{left neighbourhood $\mathcal{N}_l(B)$ of $B$} and is denoted by $\mathcal{N}_l(B)$.
    \end{enumerate}
  \end{definition}

  \begin{definition}
    Let $(X, Y, E)$ be a bipartite graph. It is called
    \begin{enumerate}
      \item \define{finite}\graffito{finite} if, and only if, the sets $X$ and $Y$ are finite.
      \item \define{locally finite}\graffito{locally finite} if, and only if, for each element $x$ of $X$, the set $\mathcal{N}_r(x)$ is finite, and for each element $y$ of $Y$, the set $\mathcal{N}_l(y)$ is finite.
    \end{enumerate}
  \end{definition}

  \begin{remark}
    Let $(X, Y, E)$ be a locally finite bipartite graph. Then, for each finite subset $A$ of $X$, the set $\mathcal{N}_r(A)$ is finite; and, for each finite subset $B$ of $Y$, the set $\mathcal{N}_l(B)$ is finite.
  \end{remark}


  \begin{definition}
    Let $(X, Y, E)$ be a bipartite graph and let $M$ be a subset of $E$. The set $M$ is called \define{matching}\graffito{matching $M$} if, and only if, for each element $(e, e')$ of $M \times M$ with $e \neq e'$, the edges $e$ and $e'$ are non-adjacent. 
  \end{definition}

  \begin{definition}
    Let $(X, Y, E)$ be a bipartite graph and let $M$ be a matching. The matching $M$ is called
    \begin{enumerate}
      \item \define{left-perfect}\graffito{left-perfect} if, and only if,
            \begin{equation*}
              \ForEach x \in X \Exists y \in Y \SuchThat (x, y) \in M; 
            \end{equation*}
      \item \define{right-perfect}\graffito{right-perfect} if, and only if,
            \begin{equation*}
              \ForEach y \in Y \Exists x \in X \SuchThat (x, y) \in M; 
            \end{equation*}
      \item \define{perfect}\graffito{perfect} if, and only if, it is left-perfect and right-perfect.
    \end{enumerate}
  \end{definition}


  \begin{definition}
    Let $(X, Y, E)$ be a locally finite bipartite graph. Is is said to satisfy the 
    \begin{enumerate}
      \item \define{left Hall condition}\graffito{left Hall condition} if, and only if, 
            \begin{equation*}
              \ForEach A \subseteq X \text{ finite} \Holds \abs{\mathcal{N}_r(A)} \geq \abs{A};
            \end{equation*}
      \item \define{right Hall condition}\graffito{right Hall condition} if, and only if,
            \begin{equation*}
              \ForEach B \subseteq Y \text{ finite} \Holds \abs{\mathcal{N}_l(B)} \geq \abs{B};
            \end{equation*}
      \item \define{Hall marriage conditions}\graffito{Hall marriage conditions} if, and only if, it satisfies the left and right Hall conditions.
    \end{enumerate}
  \end{definition}

  \begin{theorem}
    Let $(X, Y, E)$ be a locally finite bipartite graph. It satisfies the left or right Hall condition if, and only if, there is a left- or right-perfect matching, respectively.
  \end{theorem}

  \begin{proof}
    See Theorem~H.3.2 in \enquote{Cellular Automata and Groups}\cite{ceccherini-silberstein:coornaert:2010}. \qed
  \end{proof}


  \begin{theorem}
    Let $(X, Y, E)$ be a bipartite graph such, that there is a left-perfect matching and there is a right-perfect matching. There is a perfect matching.
  \end{theorem}

  \begin{proof}
    See Theorem~H.3.4 in \enquote{Cellular Automata and Groups}\cite{ceccherini-silberstein:coornaert:2010}. \qed
  \end{proof}

  \begin{corollary}[Georg Ferdinand Ludwig Philipp Cantor, Friedrich Wilhelm Karl Ernst Schröder, Felix Bernstein]
    Let $X$ and $Y$ be two sets such, that there is an injective map $f$ from $X$ to $Y$ and there is an injective map $g$ from $Y$ to $X$. There is a bijective map from $X$ to $Y$.
  \end{corollary}

  \begin{proof}
    See Theorem~H.3.5 in \enquote{Cellular Automata and Groups}\cite{ceccherini-silberstein:coornaert:2010}. \qed
  \end{proof}

  \begin{theorem}[Hall's marriage theorem, Philip Hall, 1935]
    Let $(X, Y, E)$ be a locally finite bipartite graph. It satisfies the Hall marriage conditions if, and only if, there is a perfect matching.
  \end{theorem}

  \begin{proof}
    See Theorem~H.3.6 in \enquote{Cellular Automata and Groups}\cite{ceccherini-silberstein:coornaert:2010}. \qed
  \end{proof}


  \begin{definition}
    Let $X$ and $Y$ be two sets, and let $f$ be a surjective map from $X$ to $Y$. The map $f$ is called \define{$k$-to-$1$}\graffito{$k$-to-$1$} if, and only if,
    \begin{equation*}
      \ForEach y \in Y \Holds \abs{f^{-1}(y)} = k.
    \end{equation*}
  \end{definition}

  \begin{definition}
    Let $(X, Y, E)$ be a bipartite graph, let $k$ be a positive integer, and let $M$ be a subset of $E$. The set $M$ is called \define{perfect $(1,k)$-matching}\graffito{perfect $(1,k)$-matching} if, and only if,
    \begin{equation*}
      \ForEach x \in X \Holds \abs{\set{y \in Y \suchthat (x, y) \in E}} = k
    \end{equation*}
    and
    \begin{equation*}
      \ForEach y \in Y \Holds \abs{\set{x \in X \suchthat (x, y) \in E}} = 1.
    \end{equation*}
  \end{definition}

  \begin{remark}
    The set $M$ is a perfect $(1,k)$-matching if, and only if, there is a $k$-to-$1$ surjective map $\psi \from Y \to X$ such, that $\set{(\psi(y), y) \suchthat y \in Y} = M$.
  \end{remark}

  \begin{remark} 
    The set $M$ is a perfect $(1,1)$-matching if, and only if, it is a perfect matching.
  \end{remark}

  \begin{definition}
    Let $(X, Y, E)$ be a locally finite bipartite graph and let $k$ be a positive integer. The graph $(X, Y, E)$ is said to satisfy the \define{Hall $k$-harem conditions}\graffito{Hall $k$-harem conditions} if, and only if, for each finite subset $A$ of $X$, we have $\abs{\mathcal{N}_r(A)} \geq k \abs{A}$, and for each finite subset $B$ of $Y$, we have $\abs{\mathcal{N}_l(B)} \geq k^{-1} \abs{B}$.
  \end{definition}

  \begin{theorem}[Hall's harem theorem, Philip Hall] 
  \label{thm:hall-harem}
    Let $(X, Y, E)$ be a locally finite bipartite graph and let $k$ be a positive integer. The graph $(X, Y, E)$ satisfies the Hall $k$-harem conditions if, and only if, there is a perfect $(1,k)$-matching.
  \end{theorem}

  \begin{proof}
    See Theorem~H.4.2 in \enquote{Cellular Automata and Groups}\cite{ceccherini-silberstein:coornaert:2010}. \qed
  \end{proof}

  \section{Zorn's Lemma}
  \label{apx:zorn}


  \begin{definition}
    Let $\leq$ be a preorder on $I$. It is called \define{partial order on $I$}\graffito{partial order $\leq$ on $I$} and the preordered set $(I, \leq)$ is called \define{partially ordered set}\graffito{partially ordered set $(I, \leq)$} if, and only if the relation $\leq$ is antisymmetric. 
  \end{definition}

  \begin{definition}
    Let $\leq$ be a partial order on $I$. It is called \define{total order on $I$}\graffito{total order $\leq$ on $I$} and the partially ordered set $(I, \leq)$ is called \define{totally ordered set}\graffito{totally ordered set $(I, \leq)$} if, and only if the relation $\leq$ is total. 
  \end{definition}

  \begin{definition} 
    Let $\leq$ be a preorder on $I$ and let $i$ be an element of $I$. The element $i$ is called \define{maximal in $(I, \leq)$}\graffito{maximal in $(I, \leq)$} if, and only if
    \begin{equation*}
      \ForEach i' \in I \Holds (i' \geq i \implies i' \leq i).
    \end{equation*}
  \end{definition}

  \begin{definition}
    Let $\leq$ be a preorder on $I$ and let $J$ be a subset of $I$. The set $J$ is called \define{chain in $(I, \leq)$}\graffito{chain in $(I, \leq)$} if, and only if the restriction of $\leq$ to $J$ is a total order on $J$. 
  \end{definition}

  \begin{lemma}[Zorn's Lemma]
  \label{lem:zorns-lemma}
    Let $(I, \leq)$ be a preordered set such that each chain in $I$ has an upper bound. Then, $I$ has a maximal element. \qed
  \end{lemma}



\begin{thebibliography}{9}
    \bibitem{ceccherini-silberstein:coornaert:2010} Ceccherini-Silberstein, Tullio and Michel Coornaert. Cellular Automata and Groups. In: Springer Monographs in Mathematics. Springer-Verlag, 2010.
    \bibitem{folner:1955} Følner, Erling. On groups with full Banach mean value. In: Mathematica Scandinavica 3 (1955), pages 245-254.
    \bibitem{tarski:1938} Tarski, Alfred. Algebraische Fassung des Maßproblems. In: Fundamenta Mathematicae 31.1 (1938), pages 47-66.
    \bibitem{von-neumann:1929} von Neumann, John. Zur allgemeinen Theorie des Maßes. In: Fundamenta Mathematicae 13.1 (1929), pages 73–111.
    \bibitem{wacker:automata:2016} Simon Wacker. Cellular Automata on Group Sets and the Uniform Curtis-Hedlund-Lyndon Theorem. In: Cellular Automata and Discrete Complex Systems (2016), to be published. arXiv:1603.07271 [math.GR].
    \bibitem{wacker:garden:2016} Simon Wacker. The Garden of Eden Theorem for Cellular Automata on Group Sets. arXiv:1603.07272 [math.GR].
  \end{thebibliography}
\end{document}